\documentclass[12pt]{article}

\usepackage{amssymb,amsthm,hyperref}

\newtheorem{thm}{Theorem}[section]
 
 \newtheorem{cor}{Corollary}[section]
\newtheorem{lem}{Lemma}[section]
 \newtheorem{prop}{Proposition}[section]
 \newtheorem{defn}{Definition}[section]%是否是全文计数？
\theoremstyle{remark}
\newtheorem{rem}{Remark}[section]

\title{Global classical solutions for partially dissipative hyperbolic system of balance laws}
\author{Jiang Xu\thanks {E-mail: jiangxu\underline{ }79@nuaa.edu.cn, jiangxu\underline{ }79@yahoo.com.cn}\\
\small{\textit{Department of Mathematics}},\\\small{\textit{Nanjing
University of Aeronautics and Astronautics}},
\\ \small{\textit{Nanjing 211106, P.R.China}}\\[5mm]
Shuichi Kawashima\thanks{E-mail: kawashim@math.kyushu-u.ac.jp}\\
\small{\textit{Graduate School of Mathematics}},\\
\small{\textit{Kyushu University, Fukuoka 812-8581, Japan}}}
\date{}
\begin{document}
\maketitle{} \begin{abstract} This work is concerned with
($N$-component) hyperbolic system of balance laws in arbitrary space
dimensions. Under entropy dissipative assumption and the
Shizuta-Kawashima algebraic condition, a general theory on the
well-posedness of classical solutions in the framework of
Chemin-Lerner's spaces with critical regularity is established. To
do this, we first explore the functional space theory and develop an
elementary fact that indicates the relation between homogeneous and
inhomogeneous Chemin-Lerner's spaces. Then this fact allows to prove
the local well-posedness for general data and global well-posedness
for small data by using the Fourier frequency-localization argument.
Finally, we apply the new existence theory to a specific fluid
model-the compressible Euler equations with damping, and obtain the
corresponding results in critical spaces.
\end{abstract}

\hspace{-0.5cm}\textbf{Keywords:} \small{balance laws; entropy dissipative; classical solutions; Chemin-Lerner's spaces}\\

\hspace{-0.5cm}\textbf{AMS subject classification:} \small{35L60,\
35L45,\ 35F25}

\section{Introduction}
In this work, we consider the $N$-component hyperbolic system of
balance laws, which are partial differential equations of the form
\begin{equation}
U_{t}+\sum_{j=1}^{d}F^{j}(U)_{x_{j}}=G(U). \label{R-E1}
\end{equation}
Here $U$ is the unknown $N$-vector valued function of time $t\geq0$
and space coordinate $x=(x_{1},x_{2},\cdot\cdot\cdot,x_{d})(d\geq
1)$, taking values in an open convex set $\mathcal{O}_{U}\subset
\mathbb{R}^{N}$ (the state space). $F^{j}$ and $G$ are given
$N$-vector valued smooth functions on $\mathcal{O}_{U}$. The problem
we are interested in is the Cauchy one of the system (\ref{R-E1}),
so we supplement (\ref{R-E1}) with the following initial data
\begin{equation}U_{0}=U(0,x),\ \
  x\in\mathbb{R}^{d}.\label{R-E2} \end{equation}

Note that in the absence of source term $G(U)$, (\ref{R-E1}) reduces
to a system of conservation laws. In that case, it is well-known
that classical solutions develop the singularity (e.g., shock wave)
in finite time even when the initial data are small and smooth (see,
e.g., \cite{D}). System (\ref{R-E1}) with source terms typically
govern non-equilibrium processes in physics for media with
hyperbolic response as, for example, in gas dynamics. They also
arise in the numerical simulation of conservation laws by relaxation
schemes (see \cite{AN,JX,Z} and references cited therein). In these
applications, the source term $G(U)$ has, or can be transformed by a
linear transformation into, the form
$$G(U)=\left(
 \begin{array}{c}
                                                                   0 \\
                                                                  g(U) \\
                                                                 \end{array}
                                                               \right),
$$
with $0\in \mathbb{R}^{N_{1}}, g(U)\in \mathbb{R}^{N_{2}}$, where
$N_{1}+N_{2}=N(N_{1}\neq0)$. Obviously, the dissipation is not
present in all the components of the system. A concrete example is
the compressible Euler system with damping for perfect gas flow, see
\cite{STW,WY} or Section \ref{sec:5} in this paper. As shown by
\cite{STW,WY}, the dissipative mechanisms due to the damping term,
even if it enters only in the second equation, may prevent the
formation of singularities and guarantee the global existence in
time of classical solutions, at least for some restricted classes of
initial data.

%%%%%%%%%%%%%%%%%%%%%%%%%%%%%%%%%%%%%%%%%%%%%%%%%%
Inspired by the concrete example, a natural problem is that what
conditions are posed on the general source term $G(U)$ such that it may
prevent the finite time breakdown of classical solutions for the
hyperbolic balance laws (\ref{R-E1}).
A reasonable answer is that the system (\ref{R-E1}) has an entropy
defined in \cite{KY} in a perfect manner and verifies the
Shizuta-Kawashima ([SK]) stability condition formulated in
\cite{SK}.

A notion of the entropy for (\ref{R-E1}) was first formulated by
Chen, Levermore and Liu \cite{CLL}. Their entropy was a natural
extension of the classical one due to Godunov \cite{G},
Friedrichs and Lax \cite{FL} for hyperbolic conservation laws,
i.e., (\ref{R-E1}) with $G(U)\equiv 0$, but it was not strong
enough to develop the global existence theory for (\ref{R-E1}).
Recently, under a technical requirement on the entropy dissipation
and the [SK] stability condition, Yong in \cite{Y} proved the
global existence of classical solutions in a neighborhood of a
constant equilibrium $\bar{U}\in \mathbb{R}^{N}$ satisfying
$G(\bar{U})=0$.
%Recently, in \cite{Y}, Yong constituted a suitable entropy notion
%for the balance laws (\ref{R-E1}), which was the natural extension
%of the class one due to Godunov \cite{G}, Friedrichs and Lax
%\cite{FL} for conservation laws. Under the entropy dissipation
%assumption, when the linearized symmetric system satisfies the
%Shizuta-Kawashima ([SK]) stability condition first formulated in
%\cite{SK} by the second author, it was shown that the balance laws
%(\ref{R-E1}) in arbitrary space dimensions admitted a global
%classical solution in the neighborhood of a constant equilibrium
%$\bar{U}\in \mathbb{R}^{N}$ satisfying $G(\bar{U})\equiv0$.
Hanouzet and Natalini \cite{HN} obtained a similar existence result
for one-dimensional problems in a similar situation.
For the asymptotic behavior in time of the global solutions,
%assuming the existence of a strictly convex entropy and the [SK] condition,
in a similar situation, Bianchini, Hanouzet and Natalini
\cite{BHN} claimed the solutions approach the constant equilibrium state
$\bar{U}$ in the $L^{p}$-norm at the rate
$O(t^{-\frac{d}{2}(1-\frac{1}{p})})$, as $t\rightarrow\infty$, for
$p\in[\min\{d,2\},\infty]$ by using the Duhamel principle and a
detailed analysis of the Green kernel estimates for the linearized
problem.
Subsequently, the second author and Yong \cite{KY2} removed
the technical requirement on the entropy dissipation assumed
in \cite{Y,HN,BHN} by giving a perfect definition of the entropy
for (\ref{R-E1}) and proved the same asymptotic decay estimate
as in \cite{BHN} under less regularity assumption on the initial
data. The crucial point in \cite{KY2} is to
employ the time-weighted energy method which was first developed in
\cite{Ma} for compressible Navier-Stokes equations (see also
\cite{IK}), and this enables us to show the decay estimate
for $d\geq 2$ without assuming the $L^1$ property on the initial data.
%%%%%%%%%%%%%%%%%%%%%%%%%%%%%%%%%%%%%%%%%%%%%%%%%%

It should be pointed out the above global existence and asymptotic
behavior results of classical solutions were established in the
framework of the existence theory of Kato and Majda \cite{K,M} for
generally quasi-linear hyperbolic systems (i.e.,
$\mathcal{C}_{T}H^{s}(\mathbb{R}^d)\cap
\mathcal{C}^1_{T}H^{s-1}(\mathbb{R}^d)$), where the regularity index
$\sigma$ is required to be high ($s>1+d/2$). For the case of the
critical regularity index $\sigma=1+d/2$, are there the
corresponding existence and stability for the balance laws
(\ref{R-E1})? To the best of our knowledge, this is a challenging
open problem and few results are available in this direction. In the
present paper, we shall explore the theory of functional spaces and
try to solve the open problem with the aid of the notion of entropy,
since it provides a proper setting to develop the existence theory
for the balance laws (\ref{R-E1}) in \cite{KY,Y}.

\subsection{Problem setting}
It is convenient to state basic ideas and main results of this
paper, we first review the notion of entropy and the stability
condition for (\ref{R-E1}) from \cite{KY,KY2,Y}. To begin with, we
set
\begin{eqnarray*}
\mathcal{M}=\{\psi\in\mathbb{R}^{N}: \langle\psi,G(U)\rangle=0\ \
\mbox{for any}\ U\in \mathcal{O}_{U}\},
\end{eqnarray*}
where the superscript $^{\top}$ represents the transpose. Then
$\mathcal{M}$ is a subset of $\mathbb{R}^{N}$ with $\mathrm{dim}
\mathcal{M}=N_{1}$. In the discrete kinetic theory, $\mathcal{M}$ is
called the space of summational (collision) invariants. From the
definition of $\mathcal{M}$, we have
\begin{eqnarray*}
G(U)\in\mathcal{M}^{\top}(\mbox{the orthogonal complement of}\
\mathcal{M}),\ \mbox{for any}\ U\in \mathcal{O}_{U}.
\end{eqnarray*}
Moreover, corresponding to the orthogonal decomposition
$\mathbb{R}^{N}=\mathcal{M}\oplus\mathcal{M}^{\top}$, we may write
$U\in\mathbb{R}^{N}$ as
\begin{eqnarray*}
U=\left(
    \begin{array}{c}
      U_{1} \\
      U_{2} \\
    \end{array}
  \right)
\end{eqnarray*}
such that $U\in\mathcal{M}$ holds if and only if $U_{2}=0$. We
denote by $\mathcal{E}$ the set of equilibrium state for the balance
laws (\ref{R-E1}):
\begin{eqnarray*}
\mathcal{E}=\{U\in \mathcal{O}_{U}: G(U)=0\}.
\end{eqnarray*}

In what follows, we give the notion of entropy.
\begin{defn}(\cite{KY})\label{defn1.1}
Let\ \ $\eta=\eta(U)$ be a smooth function defined in a convex open
set $\mathcal{O}_{U}\subset\mathbb{R}^{N}$. Then $\eta=\eta(U)$ is
called an entropy for the balance laws (\ref{R-E1}) if the following
statements hold:
\begin{itemize}
\item [$(\bullet)$] $\eta=\eta(U)$ is strictly convex in $\mathcal{O}_{U}$
in the sense that the Hessian $D^2_{U}\eta(U)$ is positive definite
for\ $U\in \mathcal{O}_{U}$;
\item [$(\bullet)$] $D_{U}F_{j}(U)(D^2_{U}\eta(U))^{-1}$ is symmetric for $U\in \mathcal{O}_{U}$ and $j=1,...,d;$
\item [$(\bullet)$] $U\in\mathcal{E}$ if and only if $(D_{U}\eta(U))^{\top}\in
\mathcal{M}$;
\item [$(\bullet)$] For $U\in\mathcal{E}$, the matrix
$D_{U}G(U)(D^2_{U}\eta(U))^{-1}$ is symmetric and nonpositive
definite, and its null space coincides with $\mathcal{M}$.
\end{itemize}
\end{defn}
\noindent Here and below, $D_{U}$ stands for the (row) gradient
operator with respect to $U$.

\begin{rem}

We would like to emphasize that Definition \ref{defn1.1} is a
perfect definition of the entropy for the balance laws (\ref{R-E1}),
and was introduced in \cite{KY} as a modification of the one first
formulated in \cite{CLL}. Some different definitions of entropy were
also introduced in the previous papers \cite{Y,HN,BHN}. These
definitions are, however, not good enough so that these papers have
to assume additional entropy dissipative properties such as the
property stated in Proposition \ref{prop1.1} below to get their
global existence and decay results. Note that we do not need to
assume Proposition \ref{prop1.1} for our purpose, since it directly
follows from Definition \ref{defn1.1}.
\end{rem}

Let $\eta(U)$ be the above entropy defined and set
\begin{eqnarray}W(U)=(D_{U}\eta(U))^{\top}. \label{R-E3}\end{eqnarray}
It was shown in \cite{KY} that the mapping $W=W(U)$ is a
diffeomorphism from $\mathcal{O}_{U}$ onto its range
$\mathcal{O}_{W}$. Let $U=U(W)$ be the inverse mapping which is also
a diffeomorphism from $\mathcal{O}_{W}$ onto its range
$\mathcal{O}_{U}$. Then (\ref{R-E1}) can be rewritten as

\begin{eqnarray}
A^{0}(W)W_{t}+\sum^{d}_{j=1}A^{j}(W)W_{x_{j}}=H(W) \label{R-E4}
\end{eqnarray}
with $$A^{0}(W)=D_{W}U(W),$$
$$A^{j}(W)=D_{W}F^{j}(U(W))=D_{U}F^{j}(U(W))D_{W}U(W),$$
$$H(W)=G(U(W)).$$

Moreover, let us define
$$L(W):=-D_{W}H(W)=-D_{U}G(U(W))D_{W}U(W).$$
By virtue of (\ref{R-E3}), we have
$D_{W}U(W)=D^2_{U}\eta(U(W))^{-1}$. Then it is not difficult to see
that (\ref{R-E1}) is a \textit{symmetric dissipative} system in the
sense defined as follows.

\begin{defn}(\cite{KY})\label{defn1.2}
The system (\ref{R-E4}) is called symmetric dissipative if the
following statements hold:
\begin{itemize}
\item [$(\bullet)$] $A^{0}(W)$ is symmetric and positive definite for $W\in\mathcal{O}_{W}$;
\item [$(\bullet)$] $A^{j}(W)$ is symmetric for $W\in \mathcal{O}_{W}$ and $j=1,...,d;$
\item [$(\bullet)$] $H(W)=0$ if and only if $W\in \mathcal{M}$;
\item [$(\bullet)$] For $W\in \mathcal{M}$, the matrix
$L(W)$ is symmetric and nonnegative definite, and its null space
coincides with $\mathcal{M}$.
\end{itemize}
\end{defn}

As shown by \cite{KY}, the symmetrization of balances laws can be
characterized by the existence of the entropy function.
\begin{thm}(\cite{KY})\label{thm1.1}
The following two statements are equivalent:
\begin{itemize}
\item [(i)] The System (\ref{R-E1}) has an entropy.
\item [(ii)] There is a diffeomorphism by which  (\ref{R-E1}) is
transformed to a symmetric dissipative system (\ref{R-E4}).
\end{itemize}
\end{thm}

Also, we know from \cite{KY} that the source term $H(W)$ of the
symmetric dissipative system (\ref{R-E4}) has a useful expression,
which further leads to a qualitative estimate of the entropy
production term $D_{U}\eta(U)G(U)=W^{\top}H(W)$. For clarity, we
formulate them by a proposition.

\begin{prop}(\cite{KY})\label{prop1.1}
Fixed $\bar{W}\in\mathcal{M}$. Then
$$H(W)=-LW+r(W),$$
where $L=L(\bar{W}),\ r(W)\in \mathcal{M}^{\top}$ for all
$W\in\mathcal{O}_{W}$. Furthermore, it holds that $$|r(W)|\leq
C|W-\bar{W}||(I-\mathcal{P})W|$$ and
$$\langle W, H(W)\rangle\leq-C|(I-\mathcal{P})W|^2$$ for
$W\in\mathcal{O}_{W}$ close to $\bar{W}$, where $I$ the identity
mapping on $\mathbb{R}^{N}$ and $\mathcal{P}$ the orthogonal
projection onto $\mathcal{M}$.
\end{prop}

In order to obtain the effective \textit{a priori} estimates to
extend the local solutions, we also reduce (\ref{R-E4}) to a
symmetric dissipative system of \textit{normal form} in the sense
defined below.

\begin{defn}(\cite{KY})
The symmetric dissipative system (\ref{R-E4}) is said to be of the
normal form if $A^{0}(W)$ is block-diagonal associated with
orthogonal decomposition
$\mathbb{R}^{N}=\mathcal{M}\oplus\mathcal{M}^{\bot}$.
\end{defn}

Use the partition as $$U=\left(
                           \begin{array}{c}
                             U_{1} \\
                             U_{2} \\
                           \end{array}
                         \right),\ \ \ W=\left(
                           \begin{array}{c}
                             W_{1} \\
                             W_{2} \\
                           \end{array}
                         \right)
$$
associated with orthogonal decomposition
$\mathbb{R}^{N}=\mathcal{M}\oplus\mathcal{M}^{\bot}$. We consider
the mapping $U\rightarrow V$ defined by
$$V=\left(
                           \begin{array}{c}
                             V_{1} \\
                             V_{2} \\
                           \end{array}
                         \right)=\left(
                           \begin{array}{c}
                             U_{1} \\
                             W_{2} \\
                           \end{array}
                         \right)$$
where $W_{2}=(D_{U_{2}}\eta(U))^{\top}$. This is a diffeomorphism
from $\mathcal{O}_{U}$ onto its range $\mathcal{O}_{V}$. Denote by
$U=U(V)$ the inverse mapping which is a diffeomorphism from
$\mathcal{O}_{V}$ onto its range $\mathcal{O}_{U}$. Hence, $W=W(V)$
is the diffeomorphism composed by $W=W(U)$ and $U=U(V)$. After
straightforward calculations, we show that

\begin{eqnarray}
\tilde{A}^{0}(V)V_{t}+\sum^{d}_{j=1}\tilde{A}^{j}(V)V_{x_{j}}=\tilde{H}(V)\label{R-E5}
\end{eqnarray}
with $$\tilde{A}^{0}(V)=(D_{V}W)^{\top}A^{0}(W)D_{V}W,$$
$$\tilde{A}^{j}(V)=(D_{V}W)^{\top}A^{j}(W)D_{V}W,$$
$$\tilde{H}(V)=(D_{V}W)^{\top}H(W),$$
where $W$ is evaluated at $W(V)$. Precisely,
\begin{thm}(\cite{KY2})
The system (\ref{R-E5}) is the symmetric dissipative system of the
normal form and $\tilde{H}(V)=H(W)$. It holds $W\in\mathcal{M}$ if
and only if $V\in\mathcal{M}$ between the variables $W$ with $V$.
Furthermore, the matrix $\tilde{L}(V):=-D_{V}\tilde{H}(V)$ can be
expressed as
$$\tilde{L}(V)=(D_{V}W)^{\top}L(W)D_{V}W$$ and satisfies
$\tilde{L}(V)=L(W)$ if $V\in \mathcal{M}$ (i.e., $W\in
\mathcal{M}$).
\end{thm}

As a direct consequence, we have an analogue of Proposition
\ref{prop1.1}.

\begin{cor}\label{cor1.1}
Fixed $\bar{V}\in\mathcal{M}$. Then
$$\tilde{H}(V)=-LV+\tilde{r}(V),$$
where $L=L(\bar{W}),\ \tilde{r}(V)\in \mathcal{M}^{\top}$ for all
$V\in\mathcal{O}_{V}$. Furthermore, $$|\tilde{r}(V)|\leq
C|V-\bar{V}||(I-\mathcal{P})V|.$$ and
$$\langle V,\tilde{H}(V)\rangle\leq-C|(I-\mathcal{P})V|^2$$ for
$V\in\mathcal{O}_{V}$ close to $\bar{V}$, where $I$ the identity
mapping on $\mathbb{R}^{N}$ and $\mathcal{P}$ the orthogonal
projection onto $\mathcal{M}$.
\end{cor}

Finally, we formulate the [SK] stability condition for (\ref{R-E5}),
since we deal with the symmetric dissipative system of normal form
in the subsequent analysis. Let $\bar{V}\in\mathcal{M}$ be a
constant state and consider the linearized form of (\ref{R-E5}) at
$V=\bar{V}$:

\begin{equation}
\tilde{A}^{0}V_{t}+\sum_{j=1}^{d}\tilde{A}^{j}V_{x_{j}}+LV=0,\label{R-E6}
\end{equation}
where $\tilde{A}^{0}=\tilde{A}^{0}(\bar{V}),\
\tilde{A}^{j}=\tilde{A}^{j}(\bar{V})$ and $L=L(\bar{W})$. Taking the
Fourier transform on (\ref{R-E6}) with respect to $x\in
\mathbb{R}^{d}$, we obtain
\begin{equation}
\tilde{A}^{0}\hat{V}_{t}+i|\xi|A(\omega)\hat{V}+L\hat{V}=0,
\label{R-E7}
\end{equation}
where $\tilde{A}(\omega):=\sum_{j=1}^{d}\tilde{A}^{j}\omega_{j}$
with $\omega=\xi/|\xi|\in \mathbb{S}^{d-1}$ (the unit sphere in
$\mathbb{R}^{d}$). Let $\lambda=\lambda(i\xi)$ be the eigenvalues of
(\ref{R-E7}), which solves the characteristic equation
$$\mathrm{det}(\lambda\tilde{A}^{0}+i|\xi|\tilde{A}(\omega)+L)=0.$$

Then the stability condition for (\ref{R-E6}) is stated as follows.

\begin{defn}\label{defn1.4}
The symmetric form (\ref{R-E5}) satisfies the stability condition at
$\bar{V}\in\mathcal{M}$ if the following holds true: Let $\phi\in
\mathbb{R}^{N}$ satisfies $\phi\in\mathcal{M}$ (i.e., $L\phi=0$) and
$\lambda\tilde{A}^{0}+\tilde{A}(\omega)\phi=0$ for some
$(\lambda,\omega)\in \mathbb{R}\times\mathbb{S}^{d-1}$, then
$\phi=0$.
\end{defn}

The stability condition was first formulated in \cite{SK} for
symmetric hyperbolic-parabolic coupled systems including our present
symmetric hyperbolic system (\ref{R-E4}) or (\ref{R-E5}). In
addition, the characterization of the stability condition was also
given by \cite{SK}.

\begin{thm}\label{thm1.3} The following statements are
equivalent:
\begin{itemize}
\item [$(\bullet)$] The system (\ref{R-E5}) satisfies the stability condition at $\bar{V}\in\mathcal{M}$;
\item [$(\bullet)$] $\mathrm{Re}\lambda(i\xi)<0$ for $\xi\neq0$;
\item [$(\bullet)$] There is a positive constant $c$ such that $\mathrm{Re}\lambda(i\xi)\leq-c|\xi|^2/(1+|\xi|^2)$ for $\xi\in\mathbb{R}^{d}$;
\item [$(\bullet)$] There is an $N\times N$ matrix $\tilde{K}(\omega)$ depending smooth on $\omega\in\mathbb{S}^{d-1}$ satisfying the
properties:\\(i) $\tilde{K}(-\omega)=-\tilde{K}(\omega)$ for
$\omega\in\mathbb{S}^{d-1}$;\\ (ii) $\tilde{K}(\omega)\tilde{A}^{0}$
is
skew-symmetric for $\omega\in\mathbb{S}^{d-1}$;\\
(iii)$[\tilde{K}(\omega)A(\omega)]'+L$ is positive definite for
$\omega\in\mathbb{S}^{d-1}$, where $[X]'$ denotes the symmetric

\hspace{6mm} part of the matrix $X$.
\end{itemize}
\end{thm}

\begin{rem}
We also formulate the stability condition for (\ref{R-E4}) at the
constant state $\bar{W}\in\mathcal{M}$. It turns out that the
stability condition for (\ref{R-E4}) at the constant state
$\bar{W}\in\mathcal{M}$ is equivalent to the stability condition for
(\ref{R-E5}) at $\bar{V}\in\mathcal{M}$.
\end{rem}

\subsection{Main results}
Recently, there are many well-posedness studies on the extension of
the regularity class of initial data, such as using Besov spaces, or
Triebel-Lizorkin spaces (see, e.g.\cite{C2,CH,DA,X} and references
therein). Most of those results are concerned on specific equations.
In this paper, we confine the attention to a rather general type of
equations as (\ref{R-E1})-(\ref{R-E2}), furthermore,  we study the
case of critical regularity index $(\sigma=1+d/2)$ where the
classical existence theory of Kato and Majda fails. In this
direction, there are only partial results available. In \cite{I},
Iftimie first considered (\ref{R-E5}) with $\tilde{A}^{0}(V)=I_{N}$
and gave a local existence for symmetric conservation laws
pertaining to data in the Besov space
$B^{1+d/2}_{2,1}(\mathbb{R}^d)$, which is a subalgebra embedded in
$\mathcal{C}^1(\mathbb{R}^d)$, and the lower bound of the maximal
time of existence was also obtained. Using the standard iterative
method, Chae \cite{C1} established a similar local existence for
(\ref{R-E5}) independently, where he assumed that the condition
$C^{-1}I_{N}\leq\tilde{A}^{0}(V)\leq CI_{N} (\forall V\in
\mathbb{R}^{N})$. In their works, they both considered the symmetric
conservation laws (\ref{R-E5}), i.e. without the source term
$\tilde{H}(V)$. However, up to now, the well-posedness and stability
theory for general balance laws in critical spaces still are
unknown.

The balance laws (\ref{R-E1}) with an entropy can be symmetrized,
however, the local existence of classical solutions of this paper
does not follow from the works of Iftimie and Chae \cite{I,C1}
directly. Actually, we can remove their crucial assumption
$C^{-1}I_{N}\leq\tilde{A}^{0}(V)\leq CI_{N}$, although it is
satisfied by many concrete examples. We use the classical iteration
argument and Friedrichs' regularization method to obtain the local
existence. To develop the global local existence of classical
solutions in critical spaces, the main ingredient is to construct
uniform \textit{a priori} estimates independent of time $T$
according to the dissipative mechanisms produced by source terms.
Due to the partially dissipative structure of source term $Q(U)$,
there occurs a technical obstruction. Precisely, we only capture the
dissipation rate from the partial components $(I-\mathcal{P})U$
rather than the total solutions $U$, which leads to the absence of
the low-frequency part $\Delta_{-1}(\mathcal{P}U)$ in
frequency-localization estimates. It seems that there is no chance
to obtain \textit{a priori} estimates from the standard definition
of the critical Besov space $B^{1+d/2}_{2,1}(\mathbb{R}^d)$.
Fortunately, the time-space Besov spaces (Chemin-Lerner's spaces)
$\widetilde{L}^{\rho}_{T}(B^{s}_{p,r})$ help us to overcome the
difficulty. The Chemin-Lerner's spaces were first introduced in
\cite{C2} by Chemin and Lerner, which is the refinement of the usual
spaces $L^{\rho}_{T}(B^{s}_{p,r})$. Furthermore, we explore the
functional space theory and develop a basic fact that indicates the
relation between homogeneous and inhomogeneous Chemin-Lerner's
spaces, see Proposition \ref{prop6.1} in Appendix. Then it follows
from this fact that some frequency-localization estimates in
Chemin-Lerner's spaces with critical regularity are established
effectively.

Our main results are stated as follows, where the regularity index
$\sigma=1+d/2$. First of all, we state the local well-posedness
theorem of classical solutions to the Cauchy problem
(\ref{R-E1})-(\ref{R-E2}).

\begin{thm} \label{thm1.4}
Suppose the balance laws (\ref{R-E1}) admits an entropy defined by
Definition \ref{defn1.1}. Let $\bar{U}\in \mathcal{E}$ be a constant
state. If the initial date $U_{0}$ satisfy $U_{0}-\bar{U}\in
B^{\sigma}_{2,1}(\mathbb{R}^{d})$ and take values in a compact
subset of $\mathcal{O}_{U}$, then there exists a time $T_{1}>0$ such
that
\begin{itemize}
\item[(i)] Existence:  the
Cauchy problem (\ref{R-E1})-(\ref{R-E2}) has a unique solution $U\in
\mathcal{C}^{1}([0,T_{1}]\times \mathbb{R}^{d})$ belongs to
$$U-\bar{U}\in \widetilde{\mathcal{C}}_{T_{1}}(B^{\sigma}_{2,1}(\mathbb{R}^{d}))
\cap\widetilde{\mathcal{C}}^{1}_{T_{1}}(B^{\sigma-1}_{2,1}(\mathbb{R}^{d})).
$$

\item[(ii)] Blow-up criterion: there exists a constant $C_{0}>0$
such that the maximal time $T^{*}$ of existence of such a solution
can be bounded from below by
$T^{*}\geq\frac{C_{0}}{\|U_{0}-\bar{U}\|_{B^{\sigma}_{2,1}}}.$
Moreover, if $T^{*}$ is finite, then
$$\limsup_{t\rightarrow T^{*}}\|U-\bar{U}\|_{B^{\sigma}_{2,1}}=\infty$$
if and only if $$\int^{T^{*}}_{0}\|\nabla
U\|_{L^{\infty}}dt=\infty.$$
\end{itemize}
\end{thm}

\begin{rem}
The local existence result of classical solutions holds true in the
framework of Chemin-Lerner's space with critical regularity, which
can be proved by the classical iteration argument with the help of
entropy notion, see Proposition \ref{prop3.1} for details. Let us
mention that the new general result can be regarded as an
improvement of the works of Iftimie and Chae \cite{I,C1}, which
enriches the classical local existence theory of Kato and Majda
\cite{K,M}.
\end{rem}

In small amplitude regime, with the aid of the [SK] stability
condition, we establish the global well-posedness of classical
solutions to the Cauchy problem (\ref{R-E1})-(\ref{R-E2}) in
critical spaces.

\begin{thm}\label{thm1.5}
Suppose the balance laws (\ref{R-E1}) admits an entropy defined as
Definition \ref{defn1.1} and the corresponding symmetric system
(\ref{R-E5}) satisfies the stability condition at $\bar{V}\in
\mathcal{M}$, where $\bar{V}$ is the constant state corresponding to
$\bar{U}$. There exists a positive constant $\delta_{0}$ such that
if
\begin{eqnarray*}
\|U_{0}-\bar{U}\|_{B^{\sigma}_{2,1}(\mathbb{R}^{d})}\leq \delta_{0},
\end{eqnarray*}
then the Cauchy problem (\ref{R-E1})-(\ref{R-E2}) has a unique
global solution $U\in \mathcal{C}^{1}(\mathbb{R}^{+}\times
\mathbb{R}^{d})$ satisfying
\begin{eqnarray*}
U-\bar{U}\in
\widetilde{\mathcal{C}}(B^{\sigma}_{2,1}(\mathbb{R}^{d}))\cap
\widetilde{\mathcal{C}}^1(B^{\sigma-1}_{2,1}(\mathbb{R}^{d})).
\end{eqnarray*}
Moreover, it holds that
\begin{eqnarray}
&&\|U-\bar{U}\|_{\widetilde{L}^\infty(B^{\sigma}_{2,1}(\mathbb{R}^{d}))}
+\mu_{0}\Big(\|(I-\mathcal{P})U\|_{\widetilde{L}^2(B^{\sigma}_{2,1}(\mathbb{R}^{d}))}
+\|\nabla
U\|_{\widetilde{L}^2(B^{\sigma-1}_{2,1}(\mathbb{R}^{d}))}\Big)
\nonumber\\&\leq&
C_{0}\|U_{0}-\bar{U}\|_{B^{\sigma}_{2,1}(\mathbb{R}^{d})},
\label{R-E8}
\end{eqnarray}
where $C_{0}, \mu_{0}$ are some positive constants, and
$\mathcal{P}$ is the orthogonal projection onto $\mathcal{M}$.
\end{thm}

\begin{rem}
The proof of Theorem \ref{thm1.5} relies on a crucial \textit{a
priori} estimate (Proposition \ref{prop4.1}) and the standard
continuation argument. The \textit{a priori} estimate can be done in
three steps, which is derived by Fourier frequency-localization
argument, rather than the classical energy approach as in \cite{Y}.
The first step is the basic entropy variable estimate, which leads
to the $L^2$-estimate exhibiting the dissipation rate of
$(I-\mathcal{P})U$. To take account of it, the next step is to
estimate $(I-\mathcal{P})U$ in homogeneous Chemin-Lerner's spaces
with higher space derivatives. The last step is to capture the
dissipation rate of $\nabla U$ in Fourier space, due to the
important skew-symmetry condition in Theorem \ref{thm1.3}. To
conclude, the \textit{a priori} estimate is followed by Corollary
\ref{cor6.1}.
\end{rem}

\begin{rem}
In comparison with the previous efforts in \cite{HN,KY,Y}, Theorem
\ref{thm1.5} exhibits the optimal critial regularity of the global
existence of classical solutions, which can be regarded as a
supplement to the existence theory of hyperbolic problems. On the
other hand, we see that Theorem \ref{thm1.5} is applicable to many
concrete partially dissipative balance laws, for instance, the
compressible Euler equation with damping in Sect.~\ref{sec:5}.
However, let us mention that Theorem \ref{thm1.5} was obtained by
assuming all the time the [SK] condition. As a matter of fact, this
condition is not satisfied by all physical models, such as the
equations of gas dynamics in thermal nonequilibrium (see \cite{Z}).
It would be interesting to weaken the condition while preserving the
global existence in critical spaces. This issue is under current
consideration.
\end{rem}

As a direct consequence of Theorem \ref{thm1.5}, we can see the
large-time asymptotic behavior of global solutions near the
equilibrium $\bar{U}$ in some Besov spaces.
\begin{cor}\label{cor1.2}
Let $U$ be the solution in Theorem \ref{thm1.5}. Then
$$\|\nabla \mathcal{P}U(\cdot,t)\|_{B^{\sigma-1-\varepsilon}_{2,1}(\mathbb{R}^{d})}\rightarrow
0,$$

$$\|\mathcal{P}U(\cdot,t)-\bar{U}\|_{B^{\sigma-1-\varepsilon}_{p,2}(\mathbb{R}^{d})}\rightarrow 0 \ \
\Big(p=\frac{2d}{d-2},\ d>2\Big),$$ and
$$\|(I-\mathcal{P})U(\cdot,t)\|_{B^{\sigma-\varepsilon}_{2,1}(\mathbb{R}^{d})}\rightarrow
0$$ for any  $\varepsilon>0$, as $t\rightarrow +\infty$.
\end{cor}

The rest of this paper unfolds as follows. In Sect.~\ref{sec:2}, we
introduce the Littlewood-Paley decomposition and recall the
definitions and some useful conclusions in Besov spaces and
Chemin-Lerner's spaces. In Sect.~\ref{sec:3}, we give the local
existence of classical solutions in Chemin-Lerner's spaces with
critical regularity. Sect.~\ref{sec:4} is devoted to the proof of
\textit{a priori} estimates. In Sect.~\ref{sec:5}, we present some
applications about our new results. The paper ends with an Appendix
(Sect.~\ref{sec:6}), where we develop the elementary fact that
indicates the relation between homogeneous and inhomogeneous
Chemin-Lerner's spaces, and establish the existence for linear
symmetric system, which is used to prove the local existence for the
quasilinear symmetric system (\ref{R-E5}).

\textbf{Notations}. Throughout the paper, we use
$\langle\cdot,\cdot\rangle$ to denote the standard inner product in
the real $\mathbb{R}^{N}$ or complex $\mathbb{C}^{N}$. $C>0$ stands
for a generic constant, which might be different in each context.
The notation $f\approx g$ means that $f\leq Cg$ and $g\leq Cf$.
Denote by $\mathcal{C}([0,T],X)$ (resp., $\mathcal{C}^{1}([0,T],X)$)
the space of continuous (resp., continuously differentiable)
functions on $[0,T]$ with values in a Banach space $X$. For
simplicity, the notation $\|(f,g)\|_{X}$ means $
\|f\|_{X}+\|g\|_{X}$ with $f,g\in X$. In addition, we omit the space
dependence, since all functional spaces (in $x$) are considered in
$\mathbb{R}^{d}$.

\section{Littlewood-Paley theory and functional spaces}\label{sec:2}
The proofs of most of the results presented in this paper require a
dyadic decomposition of Fourier variable, so we recall briefly the
Littlewood-Paley decomposition theory and functional spaces, such as
Besov spaces and Chemin-Lerner's spaces. The reader is also referred
to \cite{BCD} for details.

We start with the Fourier transform. The Fourier transform $\hat{f}$
of a $L^1$-function $f$ is given by
$$\mathcal{F}f=\int_{\mathbb{R}^{d}}f(x)e^{-2\pi x\cdot\xi}dx.$$ More
generally, the Fourier transform of any $f\in\mathcal{S}'$, the
space of tempered distributions, is given by
$$(\mathcal{F}f,g)=(f,\mathcal{F}g)$$ for any $g\in \mathcal{S}$, the Schwartz
class.

First, we fix some notation.
$$\mathcal{S}_{0}=\Big\{\phi\in\mathcal{S},\partial^{\alpha}\mathcal{F}f(0)=0, \forall \alpha \in \mathbb{N}^{d}\ \mbox{multi-index}\Big\}.$$
Its dual is given by
$$\mathcal{S}'_{0}=\mathcal{S}'/\mathbf{P},$$ where $\mathbf{P}$
is the space of polynomials.

We now introduce a dyadic partition of $\mathbb{R}^{d}$. We choose
$\phi_{0}\in \mathcal{S}$ such that $\phi_{0}$ is even,
$$\mathrm{supp}\phi_{0}:=A_{0}=\Big\{\xi\in\mathbb{R}^{d}:\frac{3}{4}\leq|\xi|\leq\frac{8}{3}\Big\},\  \mbox{and}\ \ \phi_{0}>0\ \ \mbox{on}\ \ A_{0}.$$
Set $A_{q}=2^{q}A_{0}$ for $q\in\mathbb{Z}$. Furthermore, we define
$$\phi_{q}(\xi)=\phi_{0}(2^{-q}\xi)$$ and define $\Phi_{q}\in
\mathcal{S}$ by
$$\mathcal{F}\Phi_{q}(\xi)=\frac{\phi_{q}(\xi)}{\sum_{q\in \mathbb{Z}}\phi_{q}(\xi)}.$$
It follows that both $\mathcal{F}\Phi_{q}(\xi)$ and $\Phi_{q}$ are
even and satisfy the following properties:
$$\mathcal{F}\Phi_{q}(\xi)=\mathcal{F}\Phi_{0}(2^{-q}\xi),\ \ \ \mathrm{supp}\ \mathcal{F}\Phi_{q}(\xi)\subset A_{q},\ \ \ \Phi_{q}(x)=2^{qd}\Phi_{0}(2^{q}x)$$
and
$$\sum_{q=-\infty}^{\infty}\mathcal{F}\Phi_{q}(\xi)=\cases{1,\ \ \ \mbox{if}\ \ \xi\in\mathbb{R}^{d}\setminus \{0\},
\cr 0, \ \ \ \mbox{if}\ \ \xi=0.}
$$
As a consequence, for any $f\in S'_{0},$ we have
$$\sum_{q=-\infty}^{\infty}\Phi_{q}\ast f=f.$$

To define the homogeneous Besov spaces, we set
$$\dot{\Delta}_{q}f=\Phi_{q}\ast f,\ \ \ \ q=0,\pm1,\pm2,...$$

\begin{defn}\label{defn2.1}
For $s\in \mathbb{R}$ and $1\leq p,r\leq\infty,$ the homogeneous
Besov spaces $\dot{B}^{s}_{p,r}$ is defined by
$$\dot{B}^{s}_{p,r}=\{f\in S'_{0}:\|f\|_{\dot{B}^{s}_{p,r}}<\infty\},$$
where
$$\|f\|_{\dot{B}^{s}_{p,r}}
=\cases{\Big(\sum_{q\in\mathbb{Z}}(2^{qs}\|\dot{\Delta}_{q}f\|_{L^p})^{r}\Big)^{1/r},\
\ r<\infty, \cr \sup_{q\in\mathbb{Z}}
2^{qs}\|\dot{\Delta}_{q}f\|_{L^p},\ \ r=\infty.} $$\end{defn}

To define the inhomogeneous Besov spaces, we set $\Psi\in
\mathcal{C}_{0}^{\infty}(\mathbb{R}^{d})$ be even and satisfy
$$\mathcal{F}\Psi(\xi)=1-\sum_{q=0}^{\infty}\mathcal{F}\Phi_{q}(\xi).$$
It is clear that for any $f\in S'_{0}$, yields
$$\Psi*f+\sum_{q=0}^{\infty}\Phi_{q}\ast f=f.$$
We further set
$$\Delta_{q}f=\cases{0,\ \ \ \ \ \ \ \, \ j\leq-2,\cr
\Psi*f,\ \ \ j=-1,\cr \Phi_{q}\ast f, \ \ j=0,1,2,...}$$

\begin{defn}\label{defn2.2}
For $s\in \mathbb{R}$ and $1\leq p,r\leq\infty,$ the inhomogeneous
Besov spaces $B^{s}_{p,r}$ is defined by
$$B^{s}_{p,r}=\{f\in S':\|f\|_{B^{s}_{p,r}}<\infty\},$$
where
$$\|f\|_{B^{s}_{p,r}}
=\cases{\Big(\sum_{q=-1}^{\infty}(2^{qs}\|\dot{\Delta}_{q}f\|_{L^p})^{r}\Big)^{1/r},\
\ r<\infty, \cr \sup_{q\geq-1} 2^{qs}\|\dot{\Delta}_{q}f\|_{L^p},\ \
r=\infty.}$$
\end{defn}

Let us point out that the definitions of $\dot{B}^{s}_{p,r}$ and
$B^{s}_{p,r}$ does not depend on the choice of the Littlewood-Paley
decomposition.  Now, we state some basic conclusions, which will be
used in subsequent analysis.
\begin{lem}(Bernstein inequality)\label{lem2.1}
Let $k\in\mathbb{N}$ and $0<R_{1}<R_{2}$. There exists a constant
$C$, depending only on $R_{1},R_{2}$ and $d$, such that for all
$1\leq a\leq b\leq\infty$ and $f\in L^{a}$,
$$
\mathrm{Supp}\mathcal{F}f\subset \{\xi\in \mathbb{R}^{d}: |\xi|\leq
R_{1}\lambda\}\Rightarrow\sup_{|\alpha|=k}\|\partial^{\alpha}f\|_{L^{b}}
\leq C^{k+1}\lambda^{k+d(\frac{1}{a}-\frac{1}{b})}\|f\|_{L^{a}};
$$
$$
\mathrm{Supp}\mathcal{F}f\subset \{\xi\in \mathbb{R}^{d}:
R_{1}\lambda\leq|\xi|\leq R_{2}\lambda\} \Rightarrow
C^{-k-1}\lambda^{k}\|f\|_{L^{a}}\leq
\sup_{|\alpha|=k}\|\partial^{\alpha}f\|_{L^{a}}\leq
C^{k+1}\lambda^{k}\|f\|_{L^{a}}.
$$
\end{lem}

As a direct corollary of the above inequality, we have
\begin{rem}\label{rem2.1} For all
multi-index $\alpha$, it holds that
$$\frac{1}{C}\|f\|_{\dot{B}^{s + |\alpha|}_{p,
r}}\leq\|\partial^\alpha f\|_{\dot{B}^s_{p, r}}\leq
C\|f\|_{\dot{B}^{s + |\alpha|}_{p, r}};$$
$$
\|\partial^\alpha f\|_{B^s_{p, r}}\leq C\|f\|_{B^{s + |\alpha|}_{p,
r}}.
$$
\end{rem}

The second one is the embedding properties in Besov spaces.
\begin{lem}\label{lem2.2} Let $s\in \mathbb{R}$ and $1\leq
p,r\leq\infty,$ then
\begin{itemize}
\item[(1)]If $s>0$, then $B^{\tilde{s}}_{p,\tilde{r}}=L^{p}\cap B^{\tilde{s}}_{p,\tilde{r}};$
\item[(2)]If $\tilde{s}\leq s$, then $B^{s}_{p,r}\hookrightarrow
B^{\tilde{s}}_{p,\tilde{r}}$. This inclusion relation is false for
the homogeneous Besov spaces;
\item[(3)]If $1\leq r\leq\tilde{r}\leq\infty$, then $\dot{B}^{s}_{p,r}\hookrightarrow
\dot{B}^{s}_{p,\tilde{r}}$ and $B^{s}_{p,r}\hookrightarrow
B^{s}_{p,\tilde{r}};$
\item[(4)]If $1\leq p\leq\tilde{p}\leq\infty$, then $\dot{B}^{s}_{p,r}\hookrightarrow \dot{B}^{s-d(\frac{1}{p}-\frac{1}{\tilde{p}})}_{\tilde{p},r}
$ and $B^{s}_{p,r}\hookrightarrow
B^{s-d(\frac{1}{p}-\frac{1}{\tilde{p}})}_{\tilde{p},r}$;
\item[(5)]$\dot{B}^{d/p}_{p,1}\hookrightarrow\mathcal{C}_{0},\ \ B^{d/p}_{p,1}\hookrightarrow\mathcal{C}_{0}(1\leq p<\infty);$
\end{itemize}
where $\mathcal{C}_{0}$ is the space of continuous bounded functions
which decay at infinity.
\end{lem}

The third one is the result of compactness in inhomogeneous Besov
spaces.
\begin{prop}\label{prop2.1}
Let $1\leq p,r\leq \infty,\ s\in \mathbb{R}$ and $\varepsilon>0$.
For all $\phi\in C_{c}^{\infty}$, the map $f\mapsto\phi f$ is
compact from $B^{s+\varepsilon}_{p,r}$ to $B^{s}_{p,r}$.
\end{prop}

On the other hand, we also present the definition of Chemin-Lerner's
space-time spaces first introduced by J.-Y. Chemin and N. Lerner
\cite{C2}, which are the refinement of the spaces
$L^{\theta}_{T}(\dot{B}^{s}_{p,r})$ or
$L^{\theta}_{T}(B^{s}_{p,r})$.

\begin{defn}\label{defn2.3}
For $T>0, s\in\mathbb{R}, 1\leq r,\theta\leq\infty$, the homogeneous
mixed time-space Besov spaces
$\widetilde{L}^{\theta}_{T}(\dot{B}^{s}_{p,r})$ is defined by
$$\widetilde{L}^{\theta}_{T}(\dot{B}^{s}_{p,r}):
=\{f\in
L^{\theta}(0,T;\mathcal{S}'_{0}):\|f\|_{\widetilde{L}^{\theta}_{T}(\dot{B}^{s}_{p,r})}<+\infty\},$$
where
$$\|f\|_{\widetilde{L}^{\theta}_{T}(\dot{B}^{s}_{p,r})}:=\Big(\sum_{q\in\mathbb{Z}}(2^{qs}\|\dot{\Delta}_{q}f\|_{L^{\theta}_{T}(L^{p})})^{r}\Big)^{\frac{1}{r}}$$
with the usual convention if $r=\infty$.
\end{defn}

\begin{defn}\label{defn2.4}
For $T>0, s\in\mathbb{R}, 1\leq r,\theta\leq\infty$, the
inhomogeneous mixed time-space Besov spaces
$\widetilde{L}^{\theta}_{T}(B^{s}_{p,r})$ is defined by
$$\widetilde{L}^{\theta}_{T}(B^{s}_{p,r}):
=\{f\in
L^{\theta}(0,T;\mathcal{S}'):\|f\|_{\widetilde{L}^{\theta}_{T}(B^{s}_{p,r})}<+\infty\},$$
where
$$\|f\|_{\widetilde{L}^{\theta}_{T}(B^{s}_{p,r})}:=\Big(\sum_{q\geq-1}(2^{qs}\|\Delta_{q}f\|_{L^{\theta}_{T}(L^{p})})^{r}\Big)^{\frac{1}{r}}$$
with the usual convention if $r=\infty$.
\end{defn}

We further define
$$\widetilde{\mathcal{C}}_{T}(B^{s}_{p,r}):=\widetilde{L}^{\infty}_{T}(B^{s}_{p,r})\cap\mathcal{C}([0,T],B^{s}_{p,r})
$$ and $$\widetilde{\mathcal{C}}^1_{T}(B^{s}_{p,r}):=\{f\in\mathcal{C}^1([0,T],B^{s}_{p,r})|\partial_{t}f\in\widetilde{L}^{\infty}_{T}(B^{s}_{p,r})\},$$
where the index $T$ will be omitted when $T=+\infty$.

Next we state some basic properties on the inhomogeneous
Chemin-Lerner's  spaces only, since the similar ones follow in the
homogeneous Chemin-Lerner's spaces.

The first one is that $\widetilde{L}^{\theta}_{T}(B^{s}_{p,r})$ may
be linked with the classical spaces $L^{\theta}_{T}(B^{s}_{p,r})$
via the Minkowski's inequality:
\begin{rem}\label{rem2.2}
It holds that
$$\|f\|_{\widetilde{L}^{\theta}_{T}(B^{s}_{p,r})}\leq\|f\|_{L^{\theta}_{T}(B^{s}_{p,r})}\,\,\,
\mbox{if}\,\, r\geq\theta;\ \ \ \
\|f\|_{\widetilde{L}^{\theta}_{T}(B^{s}_{p,r})}\geq\|f\|_{L^{\theta}_{T}(B^{s}_{p,r})}\,\,\,
\mbox{if}\,\, r\leq\theta.
$$\end{rem}
Let us also recall the property of continuity for product in
Chemin-Lerner's spaces $\widetilde{L}^{\theta}_{T}(B^{s}_{p,r})$.
\begin{prop}\label{prop2.2}
The following inequality holds:
$$
\|fg\|_{\widetilde{L}^{\theta}_{T}(B^{s}_{p,r})}\leq
C(\|f\|_{L^{\theta_{1}}_{T}(L^{\infty})}\|g\|_{\widetilde{L}^{\theta_{2}}_{T}(B^{s}_{p,r})}
+\|g\|_{L^{\theta_{3}}_{T}(L^{\infty})}\|f\|_{\widetilde{L}^{\theta_{4}}_{T}(B^{s}_{p,r})})
$$
whenever $s>0, 1\leq p\leq\infty,
1\leq\theta,\theta_{1},\theta_{2},\theta_{3},\theta_{4}\leq\infty$
and
$$\frac{1}{\theta}=\frac{1}{\theta_{1}}+\frac{1}{\theta_{2}}=\frac{1}{\theta_{3}}+\frac{1}{\theta_{4}}.$$
As a direct corollary, one has
$$\|fg\|_{\widetilde{L}^{\theta}_{T}(B^{s}_{p,r})}
\leq
C\|f\|_{\widetilde{L}^{\theta_{1}}_{T}(B^{s}_{p,r})}\|g\|_{\widetilde{L}^{\theta_{2}}_{T}(B^{s}_{p,r})}$$
whenever $s\geq d/p,
\frac{1}{\theta}=\frac{1}{\theta_{1}}+\frac{1}{\theta_{2}}.$
\end{prop}

Then we state a result of continuity for compositions in
$\widetilde{L}^{\theta}_{T}(B^{s}_{p,r})$.
\begin{prop}\label{prop2.3}
Let $s>0$, $1\leq p, r, \theta\leq \infty$, $F\in
W^{[s]+1,\infty}_{loc}(I;\mathbb{R})$ with $F(0)=0$, $T\in
(0,\infty]$ and $v\in \widetilde{L}^{\theta}_{T}(B^{s}_{p,r})\cap
L^{\infty}_{T}(L^{\infty}).$ Then
$$\|F(v)\|_{\widetilde{L}^{\theta}_{T}(B^{s}_{p,r})}\leq
C(1+\|v\|_{L^{\infty}_{T}(L^{\infty})})^{[s]+1}\|v\|_{\widetilde{L}^{\theta}_{T}(B^{s}_{p,r})}.$$
\end{prop}

Finally, we present the estimates of commutators in Chemin-Lerner's
spaces to end up this section. The indices $s,p$ behave just as in
the stationary cases (see, e.g. \cite{BCD,DA,X}) whereas the time
exponent $\theta$ behaves according to H\"{o}lder inequality.
\begin{prop}\label{prop2.4}
Let  $1<p<\infty$ and $1\leq \rho\leq\infty$. Then there exists a
generic constant $C>0$ depending only on $s, d$ such that
$$\cases{\|[f,\Delta_{q}]\mathcal{A}g\|_{L^{\theta}_{T}(L^{p})}\leq
Cc_{q}2^{-qs}\|\nabla
f\|_{\widetilde{L}^{\theta_{1}}_{T}(B^{s-1}_{p,1})}\|g\|_{\widetilde{L}^{\theta_{2}}_{T}(B^{s}_{p,1})},\
\ s=1+\frac{d}{p},\cr
\|[f,\Delta_{q}]g\|_{L^{\theta}_{T}(L^{p})}\leq
Cc_{q}2^{-q(s+1)}\|f\|_{\widetilde{L}^{\theta_{1}}_{T}(\dot{B}^{\frac{d}{p}+1}_{p,1})}\|g\|_{\widetilde{L}^{\theta_{2}}_{T}(\dot{B}^{s}_{p,1})},\
\ s\in(-\frac{d}{p}-1, \frac{d}{p}],}
$$
where the commutator $[\cdot,\cdot]$ is defined by $[f,g]=fg-gf$,
and the operator $\mathcal{A}:=\mathrm{div}$ or $\mathrm{\nabla}$.
$\{c_{q}\}$ denotes a sequence such that $\|(c_{q})\|_{ {l^{1}}}\leq
1$, $\frac{1}{\theta}=\frac{1}{\theta_{1}}+\frac{1}{\theta_{2}}$.
\end{prop}

\section{Local existence}\setcounter{equation}{0} \label{sec:3}
In this section, we prove the local existence of classical solutions
to the symmetric dissipative system (\ref{R-E5}) subject to the
following initial data
\begin{equation}
V|_{t=0}=V_{0} \ \ \ \mbox{with}\ \ \  V_{0}=V(U_{0}), \label{R-E9}
\end{equation}
which can be regarded as an improvement of the works of Iftimie
\cite{I} and Chae \cite{C1} for symmetric hyperbolic systems. Here
we could neglect the source term $\tilde{H}(V)$ in (\ref{R-E5}) for
simplicity, since it is only responsible for the large-time behavior
of solutions.

First, we consider the linear equations of (\ref{R-E5}):
\begin{equation}
\tilde{A}^{0}(V)\hat{V}_{t}+\sum_{j=1}^{d}\tilde{A}^{j}(V)\hat{V}_{x_{j}}=0,\label{R-E10}
\end{equation}
with
\begin{equation}
\hat{V}|_{t=0}=\hat{V}_{0}=V_{0}. \label{R-E11}
\end{equation}
For the initial data $V_{0}$, we assume that $V_{0}-\bar{V}\in
B^{\sigma}_{2,1}$ and
\begin{equation}
V_{0}(x)\in \mathcal{O}_{0}\ \ \ \mbox{for any}\ \ \ x\in
\mathbb{R}^{d}, \label{R-E1200}
\end{equation}
where $\mathcal{O}_{0}$ is a bounded open convex set in
$\mathbb{R}^{N}$ satisfying
$\bar{\mathcal{O}}_{0}\subset\mathcal{O}_{V}$.

For the existence of (\ref{R-E10})-(\ref{R-E11}), the reader is
referred to Proposition \ref{prop6.2} in the Appendix. Furthermore,
for $V(t,x)$, given function on $Q_{T}=[0,T]\times \mathbb{R}^{d}$,
we assume that
\begin{equation}V-\bar{V}\in \widetilde{\mathcal{C}}_{T}(B^{\sigma}_{2,1})
\cap\widetilde{\mathcal{C}}^{1}_{T}(B^{\sigma-1}_{2,1}),
\label{R-E12}
\end{equation}
\begin{equation}
V(t,x)\in\mathcal{O}_{1}\ \ \ \mbox{for  any}\ \ (t,x)\in Q_{T},
\label{R-E13}
\end{equation}
\begin{equation}
\|V(t,x)-\bar{V}\|_{\widetilde{L}^{\infty}_{T}(B^{\sigma}_{2,1})}\leq
M_{1}, \label{R-E14}
\end{equation}
\begin{equation}
\|\partial_{t}V(t,x)\|_{\widetilde{L}^{\infty}_{T}(B^{\sigma-1}_{2,1})}\leq
M_{2}, \label{R-E15}
\end{equation}
where $\mathcal{O}_{1}$ is a bounded open convex set in
$\mathbb{R}^{N}$ satisfying $\bar{\mathcal{O}}_{1}\subset
\mathcal{O}_{V},$ and $M_{1},M_{2}$ are two constants. Denote
$$X^{\sigma}_{T}(\mathcal{O}_{1};M_{1},M_{2})=\{V\in\mathcal{O}_{V}:\mbox{the conditions}\ (\ref{R-E12})-(\ref{R-E15})\ \mbox{are satisfied}\}.$$
Next, we shall prove that
$X^{\sigma}_{T}(\mathcal{O}_{1};M_{1},M_{2})$ is an invariant set
under iterations by determining $\mathcal{O}_{1},M_{1},M_{2}$ and
$T$.
\begin{lem}(invariant set under iterations)\label{lem3.1}
Suppose that the initial data satisfy $V_{0}-\bar{V}\in
B^{\sigma}_{2,1}$ and (\ref{R-E1200}). Then there exists a time
$T_{0}>0$, such that if $V\in
X^{\sigma}_{T_{0}}(\mathcal{O}_{1};M_{1},M_{2})$, the Cauchy problem
(\ref{R-E10})-(\ref{R-E11}) has a unique solution $\hat{V}$ in the
same $X^{\sigma}_{T_{0}}(\mathcal{O}_{1};M_{1},M_{2})$.
\end{lem}
\begin{proof}
First, it follows from Proposition \ref{prop6.2} and the assumption
of Lemma \ref{lem3.1} that $\hat{V}-\bar{V}\in
\widetilde{\mathcal{C}}_{T}(B^{\sigma}_{2,1})
\cap\widetilde{\mathcal{C}}^{1}_{T}(B^{\sigma-1}_{2,1})$. Next, we
show the solution $\hat{V}$ satisfies (\ref{R-E14})-(\ref{R-E15}).

Set $\hat{Z}=\hat{V}-\bar{V}$. The system (\ref{R-E10}) can be
written as
\begin{equation}
\widetilde{A}^{0}(V)\hat{Z}_{t}+\sum_{j=1}^{d}\widetilde{A}^{j}(V)\hat{Z}_{x_{j}}=0.
\label{R-E16}
\end{equation}
Applying the operator $\Delta_{q}$ to (\ref{R-E16}), we infer that
$\Delta_{q}\hat{Z}$ satisfies
\begin{equation}
\widetilde{A}^{0}(V)\Delta_{q}\hat{Z}_{t}+\sum_{j=1}^{d}\widetilde{A}^{j}(V)\Delta_{q}\hat{Z}_{x_{j}}
=-\sum_{j=1}^{d}\widetilde{A}^{0}(V)[\Delta_{q},\widetilde{A}^{0}(V)^{-1}\widetilde{A}^{j}(V)]\hat{Z}_{x_{j}},
\label{R-E17}
\end{equation}
where the commutator $[\cdot,\cdot]$ is defined by $[f,g]:=fg-gf$.

Perform the inter product with $\Delta_{q}\hat{Z}$ on both sides of
the equation (\ref{R-E17}) to get
\begin{eqnarray}
&&\langle\widetilde{A}^{0}(V)\Delta_{q}\hat{Z},\Delta_{q}\hat{Z}\rangle_{t}
+\sum_{j=1}^{d}\langle\widetilde{A}^{j}(V)\Delta_{q}\hat{Z},
\Delta_{q}\hat{Z}\rangle_{x_{j}}
\nonumber\\&=&-2\sum_{j=1}^{d}\langle\widetilde{A}^{0}(V)
[\Delta_{q},\widetilde{A}^{0}(V)^{-1}\widetilde{A}^{j}(V)]\hat{Z}_{x_{j}},\Delta_{q}\hat{Z}\rangle
+\langle\mathrm{div}\mathbb{A}(V)\Delta_{q}\hat{Z},
\Delta_{q}\hat{Z}\rangle, \label{R-E18}
\end{eqnarray}
whereafter we use the notations:
$$\mathbb{A}(V)=(\tilde{A}^{0}(V),\tilde{A}^{1}(V),\cdot\cdot\cdot,\tilde{A}^{d}(V)), \ \ \ \mathrm{div}\mathbb{A}(V)=\tilde{A}^{0}(V)_{t}+\sum_{j=1}^{d}\tilde{A}^{j}(V)_{x_{j}}.$$
Integrating (\ref{R-E18}) over $\mathbb{R}^{d}$ gives
\begin{eqnarray}
\frac{d}{dt}\|\Delta_{q}\hat{Z}\|^2_{L^2_{\tilde{A}^{0}}}&=&-2\sum_{j=1}^{d}\int_{\mathbb{R}^{d}}\langle\widetilde{A}^{0}(V)
[\Delta_{q},\widetilde{A}^{0}(V)^{-1}\widetilde{A}^{j}(V)]\hat{Z}_{x_{j}},\Delta_{q}\hat{Z}\rangle
dx \nonumber
\\&&+\int_{\mathbb{R}^{d}}\langle\mathrm{div}\mathbb{A}(V)\Delta_{q}\hat{Z},
\Delta_{q}\hat{Z}\rangle dx \label{R-E19}
\end{eqnarray}
with
$\|f\|_{L^2_{\tilde{A}^{0}}}:=\Big(\int_{\mathbb{R}^{d}}\langle\tilde{A}^{0}(V)f,f\rangle
dx\Big)^{1/2}.$

Since $V\in\mathcal{O}_{1}$ with $\bar{\mathcal{O}}_{1}\subset
\mathcal{O}_{V}$, there exists a positive constant
$C=C(\mathcal{O}_{1},M_{1})$ such that
\begin{equation}C^{-1}I_{N}\leq \tilde{A}^{0}(V)\leq CI_{N},\label{R-E20}\end{equation}
which yields $\|f\|_{L^2_{\tilde{A}^{0}}}\approx\|f\|_{L^2}$.

Therefore, we have
\begin{eqnarray}
\frac{d}{dt}\|\Delta_{q}\hat{Z}\|^2_{L^2}&=&-2\sum_{j=1}^{d}\int_{\mathbb{R}^{d}}\langle\widetilde{A}^{0}(V)
[\Delta_{q},\widetilde{A}^{0}(V)^{-1}\widetilde{A}^{j}(V)]\hat{Z}_{x_{j}},\Delta_{q}\hat{Z}\rangle
dx \nonumber
\\&&+\int_{\mathbb{R}^{d}}\langle\mathrm{div}\mathbb{A}(V)\Delta_{q}\hat{Z},
\Delta_{q}\hat{Z}\rangle dx\nonumber
\\&\leq&C\|[\Delta_{q},\widetilde{A}^{0}(V)^{-1}\widetilde{A}^{j}(V)]\hat{Z}_{x_{j}}\|_{L^2}\|\Delta_{q}\hat{Z}\|_{L^2}\\&&\nonumber+
C\|\mathrm{div}\mathbb{A}(V)\|_{L^\infty}\|\Delta_{q}\hat{Z}\|^2_{L^2}.
\label{R-E21}
\end{eqnarray}
Let $\epsilon>0$ be a small number. Dividing (\ref{R-E21}) by
$(\|\Delta_{q}\hat{Z}\|^2_{L^2}+\epsilon)^{1/2}$, we obtain
\begin{eqnarray}
\frac{d}{dt}\Big(\|\Delta_{q}\hat{Z}\|^2_{L^2}+\epsilon\Big)^{1/2}\leq
C\|[\Delta_{q},\widetilde{A}^{0}(V)^{-1}\widetilde{A}^{j}(V)]\hat{Z}_{x_{j}}\|_{L^2}+
C\|\mathrm{div}\mathbb{A}(V)\|_{L^\infty}\|\Delta_{q}\hat{Z}\|_{L^2}.
\label{R-E22}
\end{eqnarray}
A time integration yields
\begin{eqnarray}
\|\Delta_{q}\hat{Z}\|_{L^2}&\leq&\Big(\|\Delta_{q}\hat{Z}\|^2_{L^2}+\epsilon\Big)^{1/2}\nonumber
\\&\leq&\Big(\|\Delta_{q}\hat{Z}_{0}\|^2_{L^2}+\epsilon\Big)^{1/2}
+C\int^{t}_{0}\|[\Delta_{q},\widetilde{A}^{0}(V)^{-1}\widetilde{A}^{j}(V)]\hat{Z}_{x_{j}}\|_{L^2}d\tau
\nonumber
\\&&+C\int^{t}_{0}\|\mathrm{div}\mathbb{A}(V)\|_{L^\infty}\|\Delta_{q}\hat{Z}\|_{L^2}d\tau.\label{R-E23}
\end{eqnarray}
Taking the limit $\epsilon\rightarrow0$ and using the estimates of
commutators and continuity for the composition in the stationary
case (see, e.g. \cite{BCD,X}), we arrive at
\begin{eqnarray}
2^{q\sigma}\|\Delta_{q}\hat{Z}\|_{L^{\infty}_{t}(L^2)}&\leq&2^{q\sigma}\|\Delta_{q}\hat{Z}_{0}\|_{L^2}
+C\int^{t}_{0}c_{q}(\tau)\|\nabla
V\|_{B^{\sigma-1}_{2,1}}\|\hat{Z}\|_{B^{\sigma}_{2,1}}d\tau\nonumber
\\&&+C\int^{t}_{0}2^{q\sigma}\|\mathrm{div}\mathbb{A}(V)\|_{L^\infty}\|\Delta_{q}\hat{Z}\|_{L^2}d\tau ,\label{R-E24}
\end{eqnarray}
where we used Lemma \ref{lem2.2} and $\|c_{q}(t)\|_{\ell^1}\leq1$,
for all $t\in[0,T]$. Summing up (\ref{R-E24}) on $q\geq-1$ implies
\begin{eqnarray}
\|\hat{Z}\|_{\widetilde{L}^{\infty}_{T}(B^{\sigma}_{2,1})}&\leq&\|\hat{Z}_{0}\|_{B^{\sigma}_{2,1}}
+C\int^{T}_{0}(\|\nabla
V\|_{B^{\sigma}_{2,1}}+\|\mathrm{div}\mathbb{A}(V)\|_{L^\infty})\|\hat{Z}(t)\|_{B^{\sigma}_{2,1}}dt\nonumber
\\&\leq&\|\hat{Z}_{0}\|_{B^{\sigma}_{2,1}}
+C\int^{T}_{0}(\|\nabla
V\|_{B^{\sigma}_{2,1}}+\|\mathrm{div}\mathbb{A}(V)\|_{L^\infty})\|\hat{Z}\|_{\widetilde{L}^{\infty}_{t}(B^{\sigma}_{2,1})}dt.\label{R-E25}
\end{eqnarray}
Then it follows from Gronwall's inequality that
\begin{eqnarray}
\|\hat{Z}\|_{\widetilde{L}^{\infty}_{T}(B^{\sigma}_{2,1})}&\leq&
\|\hat{Z}_{0}\|_{B^{\sigma}_{2,1}}\exp\Big\{C\int^{T}_{0}(\|\nabla
V\|_{B^{\sigma}_{2,1}}+\|\mathrm{div}\mathbb{A}(V)\|_{L^\infty})dt\Big\}\nonumber
\\&\leq&e^{CT(M_{1}+M_{2})}\|V_{0}-\bar{V}\|_{B^{\sigma}_{2,1}},\label{R-E26}
\end{eqnarray}
since $V(x,t)\in X^{\sigma}_{T}(\mathcal{O}_{1};M_{1},M_{2})$. Take
$T_{0}>0$ small such that
$$e^{CT(M_{1}+M_{2})}\leq2.$$ Then the right side of (\ref{R-E26})
is bounded by $2\|V_{0}-\bar{V}\|_{B^{\sigma}_{2,1}}$, so we choose
$M_{1}=2\|V_{0}-\bar{V}\|_{B^{\sigma}_{2,1}}$.

From the equations (\ref{R-E10}), we have
\begin{eqnarray}
\hat{V}_{t}+\sum_{j=1}^{d}\{\tilde{A}^{0}(V)^{-1}\tilde{A}^{j}(V)-\tilde{A}^{0}(\bar{V})^{-1}\tilde{A}^{j}(\bar{V})\}\hat{Z}_{x_{j}}=
-\tilde{A}^{0}(\bar{V})^{-1}\tilde{A}^{j}(\bar{V})\hat{Z}_{x_{j}},\label{R-E27}
\end{eqnarray}
which implies from Proposition \ref{prop2.2} that
\begin{eqnarray}
\|\hat{V}_{t}\|_{\widetilde{L}^{\infty}_{T_{0}}(B^{\sigma-1}_{2,1})}&\leq&
C\Big(\|\tilde{A}^{0}(V)^{-1}\tilde{A}^{j}(V)-\tilde{A}^{0}(\bar{V})^{-1}\tilde{A}^{j}(\bar{V})\|_{\widetilde{L}^{\infty}_{T_{0}}(B^{\sigma}_{2,1})}+1\Big)\|\hat{Z}_{x_{j}}\|_{\widetilde{L}^{\infty}_{T_{0}}(B^{\sigma-1}_{2,1})}
\nonumber
\\&\leq&C
\Big(\|V-\bar{V}\|_{\widetilde{L}^{\infty}_{T_{0}}(B^{\sigma}_{2,1})}+1\Big)\|V-\bar{V}\|_{\widetilde{L}^{\infty}_{T_{0}}(B^{\sigma}_{2,1})}\nonumber
\\&\leq&C(1+M_{1})M_{1},\label{R-E28}
\end{eqnarray}
so we take $M_{2}=C(1+M_{1})M_{1}$.

Finally, it can be shown that the solution $\hat{V}$ satisfies
(\ref{R-E13}).
%Indeed, from the imbedding
%$B^{\sigma-1}_{2,1}\hookrightarrow \mathcal{C}_{0}$(continuous
%bounded functions which decay to zero at infinity), we arrive at
%$V_{0}-\bar{V}\in \mathcal{C}_{0}$. $\exists A>0$ such that
%$|V_{0}-\bar{V}|<1$ for all $|x|>A$. When $|x|\leq A$, we have
%$|V_{0}(x)-\bar{V}|<a$, since $V_{0}(x)-\bar{V}$ is the continuous
%function, so $V_{0}(x)\in \mathcal{O}_{0}$ with
%$\mathcal{O}_{0}=\{V\in \mathcal{O}_{V}:|V-\bar{V}|<\max(1,a)\}$
%satisfying $\bar{\mathcal{O}}_{0}\subset\mathcal{O}_{V}$.
Indeed, from $\hat{V}_t\in
\mathcal{C}([0,T_{0}],B^{\sigma-1}_{2,1})$ and the simply equality
$\hat{V}-V_{0}=\int^{t}_{0}\hat{V}_td\tau$, we deduce that
\begin{eqnarray}
|\hat{V}(t,x)-V_0(x)|\leq
T_{0}\sup_{t\in[0,T_{0}]}\|\hat{V}_t(t,\cdot)\|_{L^\infty}\leq
CT_{0}M_{2}.\label{R-E29}
\end{eqnarray}
Note that (\ref{R-E1200}), we take $T_{0}$ small so that
$CT_{0}M_{2}<d_{1}$ with
$d_{1}<d_{0}=\mbox{dist}(\mathcal{O}_{0},\partial \mathcal{O}_{V}),$
so $\hat{V}\in \mathcal{O}_{1}$ with
$\mathcal{O}_{1}:=d_{1}$-neighborhood of $\mathcal{O}_{0}$. This
completes the proof of Lemma \ref{lem3.1}.
\end{proof}

With the help of Lemma \ref{lem3.1}, we further establish the local
existence of (\ref{R-E5}) with (\ref{R-E9}).
\begin{prop}\label{prop3.1}
Assume that the initial date $V_{0}$ satisfies $V_{0}-\bar{V}\in
B^{\sigma}_{2,1}$ and (\ref{R-E1200}). Then
\begin{itemize}
\item[(i)] Existence:  there exists
a positive constant\, $T_{1}(\leq T_{0})$, depending only on
$\mathcal{O}_{0},d_{1}$ and $\|V_{0}-\bar{V}\|_{B^{\sigma}_{2,1}}$
such that the system (\ref{R-E5}) with (\ref{R-E9}) has a unique
solution $V\in X^{\sigma}_{T_{1}}(\mathcal{O}_{1};M_{1},M_{2}),$
where $\mathcal{O}_{1}, M_{1}, M_{2} $ are determined by Lemma
\ref{lem3.1}. In particular, the solution satisfies
$$V\in\mathcal{C}^{1}([0,T_{1}]\times \mathbb{R}^{d})$$
and
$$V-\bar{V}\in \widetilde{\mathcal{C}}_{T_{1}}(B^{\sigma}_{2,1})
\cap\widetilde{\mathcal{C}}^{1}_{T_{1}}(B^{\sigma-1}_{2,1}).
$$

\item[(ii)] Blow-up criterion: there exists a constant $C_{0}>0$
such that the maximal time $T^{*}$ of existence of such a solution
can be bounded from below by
$T^{*}\geq\frac{C_{0}}{\|V_{0}-\bar{V}\|_{B^{\sigma}_{2,1}}}.$
Moreover, if $T^{*}$ is finite, then
$$\limsup_{t\rightarrow T^{*}}\|V-\bar{V}\|_{B^{\sigma}_{2,1}}=\infty$$
if and only if $$\int^{T^{*}}_{0}\|\nabla
V(t)\|_{L^{\infty}}dt=\infty.$$
\end{itemize}
\end{prop}
\begin{proof}
Based on Lemma \ref{lem3.1}, we introduce the successive
approximation sequence $\{V^{n}(t,x)\}^{\infty}_{n=0}$ for the
Cauchy problem (\ref{R-E10})-(\ref{R-E11}) as follows:
$$V^{0}=\bar{V},$$
and for $n>0$,
\begin{equation}
\tilde{A}^{0}(V^{n})V^{n+1}_{t}+\sum_{j=1}^{d}\tilde{A}^{j}(V^{n})V^{n+1}_{x_{j}}=0\label{R-E30}
\end{equation}
with the initial data
\begin{equation}
V^{n+1}|_{t=0}=V_{0}(x).\label{R-E31}
\end{equation}
By Lemma \ref{lem3.1}, the sequence $\{V^{n}(t,x)\}^{\infty}_{n=0}$
is well defined on $Q_{T_{0}}$ for all $n\geq0$, and is uniformly
bounded with respect to $n\geq0$, i.e., $V^{n}\in
X^{\sigma}_{T_{0}}(\mathcal{O}_{1};M_{1},M_{2})$. Next, it will be
shown that $\{V^{n}(t,x)\}^{\infty}_{n=0}$ is a Cauchy sequence in
$\widetilde{\mathcal{C}}_{T_{0}}(B^{\sigma-1}_{2,1})$. Define
$\hat{V}^{n}=V^{n+1}-V^{n}$ for any $n\geq1$.

Take the difference between the equation (\ref{R-E30}) for the
$n+1$-th step and the $n$-th step to give
\begin{eqnarray}
&&\tilde{A}^{0}(V^{n})\hat{V}^{n}_{t}+\sum_{j=1}^{d}\tilde{A}^{j}(V^{n})\hat{V}^{n}_{x_{j}}
\nonumber\\
&&=-\tilde{A}^{0}(V^{n})\sum_{j=1}^{d}\{\tilde{A}^{0}(V^{n})^{-1}\tilde{A}^{j}(V^{n})-\tilde{A}^{0}(V^{n-1})^{-1}\tilde{A}^{j}(V^{n-1})\}V^{n}_{x_{j}}\label{R-E32}
\end{eqnarray}
subject to the initial data
\begin{equation}
\hat{V}^{n}|_{t=0}=0.\label{R-E33}
\end{equation}

Apply the operator $\Delta_{q}$ to (\ref{R-E32}) to get
\begin{eqnarray}
\tilde{A}^{0}(V^{n})\Delta_{q}\hat{V}^{n}_{t}+\sum_{j=1}^{d}\tilde{A}^{j}(V^{n})\Delta_{q}\hat{V}^{n}_{x_{j}}
=R_{1}+R_{2},\label{R-E34}
\end{eqnarray}
where
$$R_{1}:=-\sum_{j=1}^{d}\widetilde{A}^{0}(V^{n})[\Delta_{q},\widetilde{A}^{0}(V^{n})^{-1}\widetilde{A}^{j}(V^{n})]\hat{V}^{n}_{x_{j}},$$
$$R_{2}:=-\sum_{j=1}^{d}\Delta_{q}\Big\{\tilde{A}^{0}(V^{n})\Big(\tilde{A}^{0}(V^{n})^{-1}\tilde{A}^{j}(V^{n})-\tilde{A}^{0}(V^{n-1})^{-1}\tilde{A}^{j}(V^{n-1})\Big)V^{n}_{x_{j}}\Big\}.$$
Following from the similar steps as (\ref{R-E19})-(\ref{R-E23}), we
get up with
\begin{eqnarray}
\|\Delta_{q}\hat{V}^{n}\|_{L^2}\leq
C\int^{t}_{0}(\|R_{1}\|_{L^2}+\|R_{2}\|_{L^2})d\tau+C\int^{t}_{0}\|\mathrm{div}\mathbb{A}(V^{n})\|_{L^\infty}\|\Delta_{q}\hat{V}^{n}\|_{L^2}d\tau.\label{R-E35}
\end{eqnarray}
By multiplying the factor $2^{q(\sigma-1)}$ on both sides of the
inequality (\ref{R-E35}), we obtain
\begin{eqnarray}
&&2^{q(\sigma-1)}\|\Delta_{q}\hat{V}^{n}\|_{L^{\infty}_{t}(L^2)}\nonumber\\&\leq&
C\int^{t}_{0}c_{q}\|\nabla
V^{n}\|_{B^{\sigma-1}_{2,1}}\|\hat{V}^{n}\|_{B^{\sigma-1}_{2,1}}d\tau+C\int^{t}_{0}c_{q}\|
\hat{V}^{n-1}\|_{B^{\sigma-1}_{2,1}}\|V^{n}-\bar{V}\|_{B^{\sigma}_{2,1}}d\tau
\nonumber\\&&+C\int^{t}_{0}\|\mathrm{div}\mathbb{A}(V^{n})\|_{L^\infty}2^{q(\sigma-1)}\|\Delta_{q}\hat{V}^{n}\|_{L^2}d\tau,\
\ t\in[0,T_{0}], \label{R-E36}
\end{eqnarray}
where $\{c_{q}\}$ denotes some sequence which satisfies
$\|(c_{q})\|_{\ell^{1}}\leq1$.

Summing up (\ref{R-E36}) on $q\geq-1$, it is not difficult to get
\begin{eqnarray}
\|\hat{V}^{n}\|_{\widetilde{L}^{\infty}_{T_{0}}(B^{\sigma-1}_{2,1})}&\leq&
C\int^{T_{0}}_{0}\Big(\|V^{n}-\bar{V}\|_{B^{\sigma}_{2,1}}+\|\mathrm{div}\mathbb{A}(V^{n})\|_{L^\infty}\Big)\|\hat{V}^{n}\|_{\widetilde{L}^{\infty}_{t}(B^{\sigma-1}_{2,1})}dt
\nonumber\\&&+C\int^{T_{0}}_{0}\|V^{n}-\bar{V}\|_{B^{\sigma}_{2,1}}\|\hat{V}^{n-1}\|_{\widetilde{L}^{\infty}_{t}(B^{\sigma-1}_{2,1})}dt.\label{R-E37}
\end{eqnarray}
With the aid of Gronwall's inequality, we immediately deduce that
\begin{eqnarray}
&&\|\hat{V}^{n}\|_{\widetilde{L}^{\infty}_{T_{0}}(B^{\sigma-1}_{2,1})}\nonumber\\&\leq&
CT_{0}e^{C\int^{T_{0}}_{0}(\|V^{n}-\bar{V}\|_{B^{\sigma}_{2,1}}+\|\mathrm{div}\mathbb{A}(V^{n})\|_{L^\infty})dt}
\|V^{n}-\bar{V}\|_{\widetilde{L}^{\infty}_{T_{0}}(B^{\sigma}_{2,1})}\|\hat{V}^{n-1}\|_{\widetilde{L}^{\infty}_{T_{0}}(B^{\sigma-1}_{2,1})}
\nonumber\\&\leq&
CM_{1}T_{0}e^{CT_{0}(M_{1}+M_{2})}\|\hat{V}^{n-1}\|_{\widetilde{L}^{\infty}_{T_{0}}(B^{\sigma-1}_{2,1})}.\label{R-E38}
\end{eqnarray}
Take $T_{1}$ so small that
$$T_{1}\leq T_{0}\ \ \mbox{and} \ \ CM_{1}T_{1}e^{CT_{1}(M_{1}+M_{2})}\leq\frac{1}{2}.$$
Then it follows from (\ref{R-E38}) that $(V^{n}-\bar{V})$ is a
Cauchy sequence in
$\widetilde{\mathcal{C}}_{T_{1}}(B^{\sigma-1}_{2,1})$. There exists
a function $V(t,x)$ with
$V-\bar{V}\in\widetilde{\mathcal{C}}_{T_{1}}(B^{\sigma-1}_{2,1})$
such that $(V^{n}-V)\rightarrow 0$ strongly in
$\widetilde{\mathcal{C}}_{T_{1}}(B^{\sigma-1}_{2,1})$ as
$n\rightarrow\infty$. On the other hand it follows from Fatou's
property that the conditions (\ref{R-E14})-(\ref{R-E15}) are
satisfied, since $V^{n}$ is uniformly bounded in the space
$X^{\sigma}_{T_{0}}(\mathcal{O}_{1};M_{1},M_{2})\subset
X^{\sigma}_{T_{1}}(\mathcal{O}_{1};M_{1},M_{2})$. The property of
strong convergence enable us to pass to the limits in the system
(\ref{R-E30})-(\ref{R-E31}) and conclude that $V$ is a solution to
(\ref{R-E5}) and (\ref{R-E9}). What remains is to check that
$V-\bar{V}\in \mathcal{C}([0,T_{1}],B^{\sigma}_{2,1})$. Indeed, we
easily achieve that the map
$t\mapsto\|\Delta_{q}(V(t)-\bar{V})\|_{L^2}$ is continuous on
$[0,T_{1}]$, since $V-\bar{V}\in
\mathcal{C}([0,T_{1}];B^{\sigma-1}_{2,1})$. Then we have
$\Delta_{q}(V(t)-\bar{V})\in
\mathcal{C}([0,T_{1}];B^{\sigma}_{2,1})$ for all $q\geq-1$. Note
that $V-\bar{V}\in\widetilde{L}^\infty_{T_{1}}(B^{\sigma}_{2,1})$,
the series
$\sum_{q\geq-1}2^{q\sigma}\|\Delta_{q}(V(t)-\bar{V})\|_{L^2}$
converges uniformly on $[0,T_{1}]$, which yields $V-\bar{V}\in
\mathcal{C}([0,T_{1}];B^{\sigma}_{2,1})$. Moreover it follows that
$V_{t}\in\mathcal{C}([0,T_{1}];B^{\sigma-1}_{2,1})$, which implies
the condition (\ref{R-E13}) immediately. Hence, the local existence
part of solutions is complete.

Concerning the uniqueness, we set $\tilde{V}=V_{1}-V_{2}$, where
$V_{1}$ and $V_{2}$ are two solutions to the system (\ref{R-E5})
subject to the same initial data, respectively. Then the error
solution $\tilde{V}$ satisfies
\begin{eqnarray}
&&\tilde{A}^{0}(V_{1})\tilde{V}_{t}+\sum_{j=1}^{d}\tilde{A}^{j}(V_{1})\tilde{V}_{x_{j}}
\nonumber\\
&&=-\tilde{A}^{0}(V_{1})\sum_{j=1}^{d}\{\tilde{A}^{0}(V_{1})^{-1}\tilde{A}^{j}(V_{1})-\tilde{A}^{0}(V_{2})^{-1}\tilde{A}^{j}(V_{2})\}V_{2x_{j}}.\label{R-E39}
\end{eqnarray}
As above, following from the proof of Cauchy sequence, we obtain the
inequality
\begin{eqnarray}
\|\tilde{V}\|_{\widetilde{L}^{\infty}_{T_{1}}(B^{\sigma-1}_{2,1})}&\leq&
C\int^{T_{1}}_{0}\Big(\|V_{1}-\bar{V}\|_{B^{\sigma}_{2,1}}+\|V_{2}-\bar{V}\|_{B^{\sigma}_{2,1}}\Big)\|\tilde{V}\|_{\widetilde{L}^{\infty}_{t}(B^{\sigma-1}_{2,1})}dt
.\label{R-E40}
\end{eqnarray}
Gronwall's inequality implies $\tilde{V}=0$ immediately, i.e., the
uniqueness of solution $V$ is achieved. The blow-up criterion
follows from the works of Iftimie and Chae \cite{I,C1} directly, we
omit the details. This finishes the proof of Proposition
\ref{prop3.1}.
\end{proof}

\section{\textit{A priori} estimate and global existence}\setcounter{equation}{0} \label{sec:4}
To show that the classical solutions in Proposition \ref{prop3.1}
are globally defined, in this section, the central task is to
construct further \textit{a priori} estimates based on the
dissipative mechanism produced by the source term.

 To this end,
we set
$$E(T):=\|V-\bar{V}\|_{\widetilde{L}^\infty_{T}(B^{\sigma}_{2,1})}$$
and
$$D(T):=\|(I-\mathcal{P})V\|_{\widetilde{L}^2_{T}(B^{\sigma}_{2,1})}+\|\nabla
V\|_{\widetilde{L}^2_{T}(B^{\sigma-1}_{2,1})}$$ for any time $T>0$.

Our \textit{a priori} estimate reads as follows.

\begin{prop}\label{prop4.1}
If $V-\bar{V}\in \widetilde{\mathcal{C}}_{T}(B^{\sigma}_{2,1})\cap
\widetilde{\mathcal{C}}^1_{T}(B^{\sigma-1}_{2,1})$ is a solution of
(\ref{R-E5}) and (\ref{R-E9}) for any $T>0$, then there exist some
positive constants $\delta_{1}, \mu_{1}$ and $C_{1}$, if $E(T)\leq
\delta_{1}$, then
\begin{eqnarray}&&E(T)+\mu_{1}D(T)
\leq C_{1}\|V_{0}-\bar{V}\|_{B^{\sigma}_{2,1}}.\label{R-E41}
\end{eqnarray}
\end{prop}

Thanks to the standard continuation argument, we extend the
local-in-time solutions in Proposition \ref{prop3.1} with the aid of
Proposition \ref{prop4.1}, and obtain the global existence of
classical solutions to the system (\ref{R-E5}) and (\ref{R-E9}),
here we omit details,  see, e.g., \cite{MN}. Then it follows from
Theorem \ref{thm1.1} and Proposition \ref{prop2.3} that $U\in
\mathcal{C}^{1}(\mathbb{R}^{+}\times \mathbb{R}^{d})$ is the
global-in-time classical solutions of (\ref{R-E1})-(\ref{R-E2}).
Furthermore, we arrive at Theorem \ref{thm1.5}.

Actually, the proof of Proposition \ref{prop4.1} is to capture the
dissipation rates from contributions of $(I-\mathcal{P})V$ and
$\nabla V$ in turn by using the frequency-localization method. For
clarity, we divide it into several lemmas.

\begin{lem}\label{lem4.1}
If $V-\bar{V}\in \widetilde{\mathcal{C}}_{T}(B^{\sigma}_{2,1})\cap
\widetilde{\mathcal{C}}^1_{T}(B^{\sigma-1}_{2,1})$ is a solution of
(\ref{R-E5}) and (\ref{R-E9}) for any $T>0$,  then the following
estimate holds:
\begin{eqnarray}
\|V-\bar{V}\|_{L^\infty_{T}(L^2)}+\|(I-\mathcal{P})V\|_{L^2_{T}(L^2)}\leq
C\|V_{0}-\bar{V}\|_{L^2}. \label{R-E42}
\end{eqnarray}
where $C$ is a uniform positive constant independent of \ $T$.
\end{lem}

\begin{proof}
It follows from Remark \ref{rem2.2} and Lemma \ref{lem2.2} that
\begin{eqnarray}
\sup_{0\leq t\leq T}\|V(t,\cdot)-\bar{V}\|_{L^\infty}\leq CE(T)\leq
C\delta_{1}, \label{R-E43}
\end{eqnarray}
so $V(t,x)$ takes values in a neighborhood of
$\bar{V}\in\mathcal{M}$.

Let $\bar{U}\in \mathcal{E}$ be the constant state corresponding to
$\bar{V}\in \mathcal{M}$. Denote the relative entropy function
$\hat{\eta}(U)$ by
$$\hat{\eta}(U):=\eta(U)-\eta(\bar{U})-\langle D_{U}\eta(\bar{U}), U-\bar{U}\rangle.$$
Then the strictly convex quantity $\hat{\eta}(U)$ satisfies
$$\hat{\eta}(U)\geq0,\ \ \hat{\eta}(\bar{U})=0,\ \ \ D_{U}\hat{\eta}(\bar{U})=0.$$
Furthermore, $\hat{\eta}(U)$ is equivalent to the quadratic function
$|U-\bar{U}|^2$ and hence to $|V-\bar{V}|^2$, since $\hat{\eta}(U)$
is strictly convex in $U$. On the other hand, from (\ref{R-E1}), we
get the entropy-entropy flux equation
\begin{eqnarray}
\hat{\eta}(U)_{t}+\sum_{j=1}^{d}\hat{q}^{j}(U)_{x_{j}}=\langle
D_{U}\eta(U), G(U)\rangle \label{R-E44}
\end{eqnarray}
with
$\hat{q}^{j}(U)(j=1,\cdot\cdot\cdot,d)=q^{j}(U)-q^{j}(\bar{U})-\langle
D_{U}\eta(\bar{U}), F^{j}(U)-F^{j}(\bar{U})\rangle$, where
$q^{j}(U)$ is the associated entropy flux with the entropy
$\eta(U)$. It follows from (\ref{R-E3}) and Proposition
\ref{prop1.1} that
\begin{eqnarray}
\langle D_{U}\eta(U), G(U)\rangle=\langle W, H(W)\rangle\leq
-C|(I-\mathcal{P})W|^2\leq -C|(I-\mathcal{P})V|^2, \label{R-E45}
\end{eqnarray}
where we used the fact that $W\in \mathcal{M}$ is equivalent to
$V\in \mathcal{M}$. Integrating (\ref{R-E45}) over $[0,T]\times
\mathbb{R}^{d}$ implies the desired inequality (\ref{R-E42})
immediately.
\end{proof}

\begin{lem}\label{lem4.2}
If $V-\bar{V}\in \widetilde{\mathcal{C}}_{T}(B^{\sigma}_{2,1})\cap
\widetilde{\mathcal{C}}^1_{T}(B^{\sigma-1}_{2,1})$ is a solution of
(\ref{R-E5}) and (\ref{R-E9}) for any $T>0$,  then the following
estimate holds:
\begin{eqnarray}
&&\|V-\bar{V}\|_{\widetilde{L}^{\infty}_{T}(\dot{B}^{\sigma}_{2,1})}+\|(I-\mathcal{P})V\|_{\widetilde{L}^{2}_{T}(\dot{B}^{\sigma}_{2,1})}\nonumber\\&\leq&
C\|V_{0}-\bar{V}\|_{\dot{B}^{\sigma}_{2,1}}+C\sqrt{\|V-\bar{V}\|_{\widetilde{L}^{\infty}_{T}(\dot{B}^{\sigma}_{2,1})}}\Big(\|\nabla
V\|_{\widetilde{L}^2_{T}(\dot{B}^{\sigma-1}_{2,1})}\nonumber\\&&+\|(I-\mathcal{P})V\|_{\widetilde{L}^2_{T}(\dot{B}^{\sigma-1}_{2,1})}
+\|(I-\mathcal{P})V\|_{\widetilde{L}^{2}_{T}(\dot{B}^{\sigma}_{2,1})}\Big).
\label{R-E46}
\end{eqnarray}
where $C$ is a uniform positive constant independent of \ $T$.
\end{lem}
\begin{proof}
Indeed, from Corollary \ref{cor1.1}, the normal form (\ref{R-E5})
can be written as
\begin{eqnarray}
\tilde{A}^{0}(V)V_{t}+\sum_{j=1}^{d}\tilde{A}^{j}(V)V_{x_{j}}+LV=\tilde{r}(V).\label{R-E47}
\end{eqnarray}
Applying the homogeneous localization operator
$\dot{\Delta}_{q}(q\in \mathbb{Z})$ to (\ref{R-E47}), we obtain
\begin{eqnarray}
\tilde{A}^{0}(V)\dot{\Delta}_{q}V_{t}+\sum_{j=1}^{d}\tilde{A}^{j}(V)\dot{\Delta}_{q}V_{x_{j}}+L\dot{\Delta}_{q}V=I_{1}(V)+I_{2}(V),\label{R-E48}
\end{eqnarray}
with
$$
I_{1}(V)=-\sum_{j=1}^{d}\tilde{A}^{0}(V)[\dot{\Delta}_{q},\tilde{A}^{0}(V)^{-1}\tilde{A}^{j}(V)]V_{x_{j}},
$$ and
$$
I_{2}(V)=\tilde{A}^{0}(V)\Big\{-[\dot{\Delta}_{q},\tilde{A}^{0}(V)^{-1}]LV+\dot{\Delta}_{q}\Big(\tilde{A}^{0}(V)^{-1}\tilde{r}(V)\Big)\Big\},
$$
where the commutator $[\cdot,\cdot]$ is defined by $[f,g]:=fg-gf$.
Take the inner product of (\ref{R-E48}) with $\dot{\Delta}_{q}V$ to
get
\begin{eqnarray}
&&\langle\tilde{A}^{0}(V)\dot{\Delta}_{q}V,\dot{\Delta}_{q}V\rangle_{t}+\sum_{j=1}^{d}\langle\tilde{A}^{j}(V)\dot{\Delta}_{q}V,\dot{\Delta}_{q}V\rangle_{x_{j}}
+2\langle L\dot{\Delta}_{q}V,\dot{\Delta}_{q}V\rangle\nonumber\\
&&\hspace{20mm}=2\langle(I_{1}(V)+I_{2}(V),\dot{\Delta}_{q}V\rangle+\langle\mathrm{div}\mathbb{A}(V)\dot{\Delta}_{q}V,\dot{\Delta}_{q}V\rangle,
\label{R-E49}
\end{eqnarray}
where
$$ \mathrm{div}\mathbb{A}(V)=\tilde{A}^{0}(V)_{t}+\sum_{j=1}^{d}\tilde{A}^{j}(V)_{x_{j}}$$
and
$$\|f\|_{L^2_{\tilde{A}^{0}}}:=\Big(\int_{\mathbb{R}^{d}}\langle\tilde{A}^{0}(V)f,f\rangle dx\Big)^{1/2},\
\
\|f\|_{\widetilde{L}^{2}_{T}(\dot{\mathcal{B}}^{\sigma}_{2,1})}=\|2^{q\sigma}\|\dot{\Delta}_{q}f\|_{L^2_{T}(L^2_{\tilde{A}^{0}})}\|_{\ell^{1}(\mathbb{Z})}.$$
Obviously, the norms
$\|f\|_{L^2_{\tilde{A}^{0}}}\thickapprox\|f\|_{L^2}$ and
$\|f\|_{\widetilde{L}^{2}_{T}(\dot{\mathcal{B}}^{\sigma}_{2,1})}\thickapprox\|f\|_{\widetilde{L}^{2}_{T}(\dot{B}^{\sigma}_{2,1})}$,
since $V(t,x)$ takes values in a neighborhood of
$\bar{V}\in\mathcal{M}$.

Integrating (\ref{R-E49}) over $[0,t]\times \mathbb{R}^{d}\
(t\in[0,T])$ gives
\begin{eqnarray}
&&\|\dot{\Delta}_{q}(V-\bar{V})\|^2_{L^2_{\tilde{A}^{0}}}\Big|^{t}_{0}+\|(I-\mathcal{P})\dot{\Delta}_{q}V\|^2_{L^2_{t}(L^2)}
\nonumber\\
&=&C\int_{\mathbb{R}^{+}}\int_{\mathbb{R}^{d}}\langle
I_{1}(V)+I_{2}(V),\dot{\Delta}_{q}V\rangle dxdt\nonumber\\&&
+C\int_{\mathbb{R}^{+}}\int_{\mathbb{R}^{d}}\langle\mathrm{div}\mathbb{A}(V)\dot{\Delta}_{q}V,\dot{\Delta}_{q}V\rangle
dxdt,\label{R-E50}
\end{eqnarray}
where we used the fact $\dot{\Delta}_{q}\bar{V}=0$. Next, we begin
to estimate the integrations in the right-hand side of (\ref{R-E50})
respectively.

From the equations (\ref{R-E5}) and Corollary \ref{cor1.1}, we have
$|V_{t}|\leq C(|\nabla V|+|(I-\mathcal{P})V|)$. Then
\begin{eqnarray}
&&\Big|\int_{\mathbb{R}^{+}}\int_{\mathbb{R}^{d}}\langle\mathrm{div}\mathbb{A}(V)\dot{\Delta}_{q}V,\dot{\Delta}_{q}V\rangle
dxdt\Big|\nonumber\\&\leq&
C\int_{\mathbb{R}^{+}}(\|V_{t}\|_{L^\infty}+\|\nabla
V\|_{L^\infty})\|\dot{\Delta}_{q}V\|^2_{L^2}dt\nonumber\\&\leq& C
\int_{\mathbb{R}^{+}}(\|\nabla
V\|_{L^\infty}+\|(I-\mathcal{P})V\|_{L^\infty})\|\dot{\Delta}_{q}V\|^2_{L^2}dt
\nonumber\\&\leq& C\Big(\|\nabla
V\|_{L^2_{t}(L^\infty)}+\|(I-\mathcal{P})V\|_{L^2_{t}(L^\infty)}\Big)\|\dot{\Delta}_{q}V\|_{L^{\infty}_{t}(L^2)}\|\dot{\Delta}_{q}V\|_{L^{2}_{t}(L^2)}
\nonumber\\&\leq& Cc_{q}^22^{-2q\sigma}\Big(\|\nabla
V\|_{L^2_{t}(\dot{B}^{\sigma-1}_{2,1})}+\|(I-\mathcal{P})V\|_{L^2_{t}(\dot{B}^{\sigma-1}_{2,1})}\Big)\nonumber\\&&\times
\|V-\bar{V}\|_{\widetilde{L}^{\infty}_{t}(\dot{B}^{\sigma}_{2,1})}\|\nabla
V\|_{\widetilde{L}^{2}_{t}(\dot{B}^{\sigma-1}_{2,1})},\label{R-E51}
\end{eqnarray}
where we used the embedding property (5) of Lemma \ref{lem2.2}. Here
and below $\{c_{q}\}$ denotes some sequence which satisfies
$\|(c_{q})\|_{ {l^{1}(\mathbb{Z})}}\leq 1$, although each
$\{c_{q}\}$ is possibly different.

With the aid of Cauchy-Schwartz inequality and the commutator
estimates in Proposition \ref{prop2.4}, we have
\begin{eqnarray}
&&\int_{\mathbb{R}^{+}}\int_{\mathbb{R}^{d}}\langle
I_{1}(V),\dot{\Delta}_{q}V\rangle dxdt \nonumber\\&\leq& C
\sum_{j=1}^{d}\|[\dot{\Delta}_{q},\tilde{A}^{0}(V)^{-1}\tilde{A}^{j}(V)]V_{x_{j}}\|_{L^2_{t}(L^2)}\|\dot{\Delta}_{q}V\|_{L^{2}_{t}(L^2_{\tilde{A}^{0}})}
\nonumber\\&\leq&
Cc_{q}2^{-q\sigma}\|\tilde{A}^{0}(V)^{-1}\tilde{A}^{j}(V)\|_{\widetilde{L}^{\infty}_{t}(\dot{B}^{\sigma}_{2,1})}\|\nabla
V\|_{\widetilde{L}^{2}_{t}(\dot{B}^{\sigma-1}_{2,1})}\|\dot{\Delta}_{q}V\|_{L^{2}_{t}(L^2_{\tilde{A}^{0}})}
\nonumber\\&\leq&
Cc_{q}2^{-q\sigma}\|V-\bar{V}\|_{\widetilde{L}^{\infty}_{t}(\dot{B}^{\sigma}_{2,1})}\|\nabla
V\|_{\widetilde{L}^{2}_{t}(\dot{B}^{\sigma-1}_{2,1})}\|\dot{\Delta}_{q}V\|_{L^{2}_{t}(L^2_{\tilde{A}^{0}})}\label{R-E52}
\end{eqnarray}
and
\begin{eqnarray}
&&\int_{\mathbb{R}^{+}}\int_{\mathbb{R}^{d}}\langle
I_{2}(V),\dot{\Delta}_{q}V\rangle dxdt \nonumber\\&\leq&
C\Big(\|[\dot{\Delta}_{q},\tilde{A}^{0}(V)^{-1}]LV\|_{L^2_{t}(L^2)}+\|\dot{\Delta}_{q}(\tilde{A}^{0}(V)^{-1}\tilde{r}(V))\|_{L^2_{t}(L^2)}\Big)\nonumber\\&&\times\|(I-\mathcal{P})\dot{\Delta}_{q}V\|_{L^{2}_{t}(L^2_{\tilde{A}^{0}})}
\nonumber\\&\leq&Cc_{q}2^{-q\sigma}\Big(\|\tilde{A}^{0}(V)^{-1}\|_{\widetilde{L}^{\infty}_{t}(\dot{B}^{\sigma}_{2,1})}\|LV\|_{\widetilde{L}^{2}_{t}(\dot{B}^{\sigma-1}_{2,1})}
+\|\tilde{A}^{0}(V)^{-1}\tilde{r}(V)\|_{\widetilde{L}^{2}_{t}(\dot{B}^{\sigma}_{2,1})}\Big)\nonumber\\&&\times\|(I-\mathcal{P})\dot{\Delta}_{q}V\|_{L^{2}_{t}(L^2_{\tilde{A}^{0}})}
\nonumber\\&\leq&Cc_{q}2^{-q\sigma}\|V-\bar{V}\|_{\widetilde{L}^{\infty}_{t}(\dot{B}^{\sigma}_{2,1})}
\Big(\|(I-P)V\|_{\widetilde{L}^{2}_{t}(\dot{B}^{\sigma-1}_{2,1})}
+\|\nabla
V\|_{\widetilde{L}^{2}_{t}(\dot{B}^{\sigma-1}_{2,1})}\Big)\nonumber\\&&\times\|(I-\mathcal{P})\dot{\Delta}_{q}V\|_{L^{2}_{t}(L^2_{\tilde{A}^{0}})},\label{R-E53}
\end{eqnarray}
where we also used the homogeneous versions as Propositions
\ref{prop2.2}-\ref{prop2.3}.

Combining with (\ref{R-E50})-(\ref{R-E53}), we are led to the
estimate
\begin{eqnarray}
&&\|\dot{\Delta}_{q}(V-\bar{V})\|^2_{L^2_{\tilde{A}^{0}}}+\|(I-\mathcal{P})\dot{\Delta}_{q}V\|^2_{L^2_{t}(L^2)}
\nonumber\\&\leq&
\|\dot{\Delta}_{q}(V_{0}-\bar{V})\|^2_{L^2_{\tilde{A}^{0}}}+Cc_{q}^22^{-2q\sigma}\Big(\|\nabla
V\|_{\widetilde{L}^2_{t}(\dot{B}^{\sigma-1}_{2,1})}+\|(I-\mathcal{P})V\|_{\widetilde{L}^2_{t}(\dot{B}^{\sigma-1}_{2,1})}\Big)\nonumber\\&&\times
\|V-\bar{V}\|_{\widetilde{L}^{\infty}_{t}(\dot{B}^{\sigma}_{2,1})}\|\nabla
V\|_{\widetilde{L}^{2}_{t}(\dot{B}^{\sigma-1}_{2,1})}+Cc_{q}2^{-q\sigma}\|V-\bar{V}\|_{\widetilde{L}^{\infty}_{t}(\dot{B}^{\sigma}_{2,1})}\|\nabla
V\|_{\widetilde{L}^{2}_{t}(\dot{B}^{\sigma-1}_{2,1})}\nonumber\\&&\times\|\dot{\Delta}_{q}V\|_{L^{2}_{t}(L^2_{\tilde{A}^{0}})}+
Cc_{q}2^{-q\sigma}\|V-\bar{V}\|_{\widetilde{L}^{\infty}_{t}(\dot{B}^{\sigma}_{2,1})}
\Big(\|(I-\mathcal{P})V\|_{\widetilde{L}^{2}_{t}(\dot{B}^{\sigma-1}_{2,1})}
\nonumber\\&&+\|\nabla
V\|_{\widetilde{L}^{2}_{t}(\dot{B}^{\sigma-1}_{2,1})}\Big)\|(I-\mathcal{P})\dot{\Delta}_{q}V\|_{L^{2}_{t}(L^2_{\tilde{A}^{0}})}.\label{R-E54}
\end{eqnarray}
Then multiplying the factor $2^{2q\sigma}$ on both sides of
(\ref{R-E54}), we obtain
\begin{eqnarray}
&&2^{2q\sigma}\|\dot{\Delta}_{q}(V-\bar{V})\|^2_{L^2_{\tilde{A}^{0}}}+2^{2q\sigma}\|(I-\mathcal{P})\dot{\Delta}_{q}V\|^2_{L^2_{t}(L^2)}
\nonumber\\&\leq&
2^{2q\sigma}\|\dot{\Delta}_{q}(V_{0}-\bar{V})\|^2_{L^2_{\tilde{A}^{0}}}+Cc_{q}^2\Big(\|\nabla
V\|_{\widetilde{L}^2_{t}(\dot{B}^{\sigma-1}_{2,1})}+\|(I-\mathcal{P})V\|_{\widetilde{L}^2_{t}(\dot{B}^{\sigma-1}_{2,1})}\Big)\nonumber\\&&\times
\|V-\bar{V}\|_{\widetilde{L}^{\infty}_{t}(\dot{B}^{\sigma}_{2,1})}\|\nabla
V\|_{\widetilde{L}^{2}_{t}(\dot{B}^{\sigma-1}_{2,1})}+Cc^2_{q}\|V-\bar{V}\|_{\widetilde{L}^{\infty}_{t}(\dot{B}^{\sigma}_{2,1})}\|\nabla
V\|_{\widetilde{L}^{2}_{t}(\dot{B}^{\sigma-1}_{2,1})}\nonumber\\&&\times\|\nabla
V\|_{\widetilde{L}^{2}_{t}(\dot{\mathcal{B}}^{\sigma-1}_{2,1})}+
Cc^2_{q}\|V-\bar{V}\|_{\widetilde{L}^{\infty}_{t}(\dot{B}^{\sigma}_{2,1})}
\Big(\|(I-\mathcal{P})V\|_{\widetilde{L}^{2}_{t}(\dot{B}^{\sigma-1}_{2,1})}
\nonumber\\&&+\|\nabla
V\|_{\widetilde{L}^{2}_{t}(\dot{B}^{\sigma-1}_{2,1})}\Big)\|(I-\mathcal{P})V\|_{\widetilde{L}^{2}_{t}(\dot{\mathcal{B}}^{\sigma}_{2,1})}.\label{R-E55}
\end{eqnarray}
By performing Young's inequality, we have
\begin{eqnarray}
&&2^{q\sigma}\|\dot{\Delta}_{q}(V-\bar{V})\|_{L^{\infty}_{T}(L^2)}+2^{q\sigma}\|(I-\mathcal{P})\dot{\Delta}_{q}V\|_{L^2_{T}(L^2)}
\nonumber\\&\leq&
2^{q\sigma}\|\dot{\Delta}_{q}(V_{0}-\bar{V})\|_{L^2}
+Cc_{q}\sqrt{\|V-\bar{V}\|_{\widetilde{L}^{\infty}_{T}(\dot{B}^{\sigma}_{2,1})}}\nonumber\\&&\times\Big(\|\nabla
V\|_{\widetilde{L}^2_{T}(\dot{B}^{\sigma-1}_{2,1})}+\|(I-\mathcal{P})V\|_{\widetilde{L}^2_{T}(\dot{B}^{\sigma-1}_{2,1})}
+\|(I-\mathcal{P})V\|_{\widetilde{L}^{2}_{T}(\dot{B}^{\sigma}_{2,1})}\Big),\label{R-E56}
\end{eqnarray}
where we used norms
$\|f\|_{L^2_{\tilde{A}^{0}}}\thickapprox\|f\|_{L^2}$ and
$\|f\|_{\widetilde{L}^{2}_{T}(\dot{\mathcal{B}}^{\sigma}_{2,1})}\approx\|f\|_{\widetilde{L}^{2}_{T}(\dot{B}^{\sigma}_{2,1})}.$

Summing up (\ref{R-E56}) on $q\in\mathbb{Z}$, we can get
(\ref{R-E46}) immediately.
\end{proof}

\begin{rem}\label{rem 4.1}
According to the elementary fact
$$L^{\theta}(L^{p})\cap \widetilde{L}^{\theta}(\dot{B}^{s}_{p,r})=\widetilde{L}^{\theta}(B^{s}_{p,r})(\theta\geq r),$$
and Corollary \ref{cor6.1},  which are achieved in the Appendix, it
follows from (\ref{R-E42}) and (\ref{R-E46}) that
\begin{eqnarray}
E(T)+\|(I-\mathcal{P})V\|_{\widetilde{L}^{2}_{T}(B^{\sigma}_{2,1})}\leq
C\|V_{0}-\bar{V}\|_{B^{\sigma}_{2,1}}+C\sqrt{E(T)}D(T).\label{R-E57}
\end{eqnarray}
\end{rem}

\begin{lem}\label{lem4.3}
If $V-\bar{V}\in \widetilde{\mathcal{C}}_{T}(B^{\sigma}_{2,1})\cap
\widetilde{\mathcal{C}}^1_{T}(B^{\sigma-1}_{2,1})$ is a solution of
(\ref{R-E5}) and (\ref{R-E9}) for any $T>0$,  then the following
estimate holds:
\begin{eqnarray}
\|\nabla V\|_{\widetilde{L}^2_{T}(B^{\sigma-1}_{2,1})}&\leq&
C\|V_{0}-\bar{V}\|_{B^{\sigma}_{2,1}}+CE(T)\nonumber\\&&+C\|(I-\mathcal{P})V\|_{\widetilde{L}^{2}_{T}(B^{\sigma}_{2,1})}+C\sqrt{E(T)}D(T),\label{R-E58}
\end{eqnarray}
where $C$ is a uniform positive constant independent of \ $T$.
\end{lem}

\begin{proof}
In this paragraph, we shall take full advantage of the [SK]
stability condition in Theorem \ref{thm1.3} to establish the
inequality (\ref{R-E58}).

For this purpose, we set $Z=V-\bar{V}$. Then the symmetric
hyperbolic system (\ref{R-E5}) can be written as the linearized form
at $\bar{V}$
\begin{eqnarray}
\tilde{A}^{0}Z_{t}+\sum^{d}_{j=1}\tilde{A}^{j}Z_{x_{j}}+LZ=\mathcal{G},\label{R-E59}
\end{eqnarray}
where $\tilde{A}^{0}=\tilde{A}^{0}(\bar{V}),
\tilde{A}^{j}=\tilde{A}^{j}(\bar{V})$ and $L$ are the constant
matrices, and $\mathcal{G}=\mathcal{G}_{1}+\mathcal{G}_{2}$ with
$$\mathcal{G}_{1}:=-\sum^{d}_{j=1}\tilde{A}^{0}\Big(\tilde{A}^{0}(V)^{-1}\tilde{A}^{j}(V)-(\tilde{A}^{0})^{-1}\tilde{A}^{j}\Big)V_{x_{j}},$$
$$\mathcal{G}_{2}:=\tilde{A}^{0}\Big\{-\Big(\tilde{A}^{0}(V)^{-1}-(\tilde{A}^{0})^{-1}\Big)LV+\tilde{A}^{0}(V)^{-1}\tilde{r}(V)\Big\}.$$
Applying the inhomogeneous localization operator
$\Delta_{q}(q\geq-1)$ to  (\ref{R-E59}) gives
\begin{eqnarray}
\tilde{A}^{0}\Delta_{q}Z_{t}+\sum^{d}_{j=1}\tilde{A}^{j}\Delta_{q}Z_{x_{j}}+L\Delta_{q}Z=\Delta_{q}\mathcal{G}.\label{R-E60}
\end{eqnarray}
Perform the Fourier transform with respect to the space variable $x$
for (\ref{R-E60}) and multiply the resulting equation by
$-i|\xi|\tilde{K}(\omega)$. Then by taking the inner product with
$\widehat{\Delta_{q}Z}$ and taking the real part of each term in the
resulting equality, we can obtain
\begin{eqnarray}
&&|\xi|\mathrm{Im}\Big\langle\tilde{K}(\omega)\tilde{A}^{0}\frac{d}{dt}\widehat{\Delta_{q}Z},\widehat{\Delta_{q}Z}\Big\rangle+|\xi|^2
\langle[\tilde{K}(\omega)\tilde{A}(\omega)]'\widehat{\Delta_{q}Z},\widehat{\Delta_{q}Z}\rangle\nonumber\\&&\hspace{6mm}=
-|\xi|\mathrm{Im}\langle\tilde{K}(\omega)L\widehat{\Delta_{q}Z},\widehat{\Delta_{q}Z}\rangle
+|\xi|\mathrm{Im}\langle\tilde{K}(\omega)\widehat{\Delta_{q}\mathcal{G}},\widehat{\Delta_{q}Z}\rangle
\label{R-E61}
\end{eqnarray}
with $\tilde{A}(\omega)=\sum^{d}_{j=1}\tilde{A}^{j}\omega_{j}$ and
$\omega=\xi/|\xi|.$

The skew-symmetry of $\tilde{K}(\omega)\tilde{A}^{0}$ implies that
\begin{eqnarray}
|\xi|\mathrm{Im}\Big\langle\tilde{K}(\omega)\tilde{A}^{0}\frac{d}{dt}\widehat{\Delta_{q}Z},\widehat{\Delta_{q}Z}\Big\rangle=
\frac{1}{2}\frac{d}{dt}|\xi|\mathrm{Im}\langle\tilde{K}(\omega)\tilde{A}^{0}\widehat{\Delta_{q}Z},\widehat{\Delta_{q}Z}\rangle.\label{R-E62}
\end{eqnarray}
It follows from the condition that
$[\tilde{K}(\omega)\tilde{A}(\omega)]'+L $ is positive definite for
$\omega\in \mathbb{S}^{d-1}$ that
\begin{eqnarray}
|\xi|^2
\langle[\tilde{K}(\omega)\tilde{A}(\omega)]'\widehat{\Delta_{q}Z},\widehat{\Delta_{q}Z}\rangle\geq
c|\xi|^2|\widehat{\Delta_{q}Z}|^2-C|\xi|^2|(I-\mathcal{P})\widehat{\Delta_{q}Z}|^2.\label{R-E63}
\end{eqnarray}
Moreover, by virtue of Young's inequality, we have
\begin{eqnarray}
\Big||\xi|\mathrm{Im}\langle\tilde{K}(\omega)L\widehat{\Delta_{q}Z},\widehat{\Delta_{q}Z}\rangle\Big|\leq\varepsilon|\xi|^2|\widehat{\Delta_{q}Z}|^2
+C_{\varepsilon}|(I-\mathcal{P})\widehat{\Delta_{q}Z}|^2\label{R-E64}
\end{eqnarray}
for any $\varepsilon>0$, and
\begin{eqnarray}
\Big||\xi|\mathrm{Im}\langle\tilde{K}(\omega)\widehat{\Delta_{q}\mathcal{G}},\widehat{\Delta_{q}Z}\rangle\Big|\leq
C|\xi||\widehat{\Delta_{q}Z}||\widehat{\Delta_{q}\mathcal{G}}|,\label{R-E65}
\end{eqnarray}
where we have used the uniform boundedness of the matrix
$\tilde{K}(\omega)$. Combining the equalities
(\ref{R-E61})-(\ref{R-E62}) and the inequalities
(\ref{R-E63})-(\ref{R-E65}), we deduce that
\begin{eqnarray}
|\xi|^2|\widehat{\Delta_{q}Z}|^2&\leq&
C(1+|\xi|^2)|(I-\mathcal{P})\widehat{\Delta_{q}Z}|^2+C|\xi||\widehat{\Delta_{q}Z}||\widehat{\Delta_{q}\mathcal{G}}|\nonumber\\&&-
\frac{1}{2}\frac{d}{dt}|\xi|\mathrm{Im}\langle\tilde{K}(\omega)\tilde{A}^{0}\widehat{\Delta_{q}Z},\widehat{\Delta_{q}Z}\rangle.\label{R-E66}
\end{eqnarray}
Integrating (\ref{R-E66}) over $[0,t]\times\mathbb{R}^{d}$,
 and using Plancherel's theorem yields
\begin{eqnarray}
&&\int^t_{0}\|\Delta_{q}\nabla Z\|^2_{L^2}d\tau\nonumber\\&\leq
&C\int^t_{0}\Big(\|(I-\mathcal{P})\Delta_{q}
Z\|^2_{L^2}+\|(I-\mathcal{P})\Delta_{q}\nabla
Z\|^2_{L^2}\Big)d\tau+C\int^t_{0}\|\Delta_{q}\nabla
Z\|_{L^2}\|\Delta_{q}\mathcal{G}\|_{L^2}d\tau\nonumber\\&&-
\frac{1}{2}\frac{d}{dt}|\xi|\mathrm{Im}\langle\tilde{K}(\omega)\tilde{A}^{0}\widehat{\Delta_{q}Z},\widehat{\Delta_{q}Z}\rangle
d\xi\Big|^t_{0}.\label{R-E67}
\end{eqnarray}
The last term on the right side of (\ref{R-E67}) can be estimated as
\begin{eqnarray}
&&-\frac{1}{2}\frac{d}{dt}|\xi|\mathrm{Im}\langle\tilde{K}(\omega)\tilde{A}^{0}\widehat{\Delta_{q}Z},\widehat{\Delta_{q}Z}\rangle
d\xi\Big|^t_{0} \nonumber\\&\leq &
C\int_{\mathbb{R}^{d}}(1+|\xi|^2)(|(\widehat{\Delta_{q}Z})|^2+|(\widehat{\Delta_{q}Z_{0}})|^2)d\xi\nonumber\\&\leq
&C2^{2q}(\|\Delta_{q}Z\|^2_{L^2}+\|\Delta_{q}Z_{0}\|^2_{L^2}),\label{R-E68}
\end{eqnarray}
where we used Bernstein's inequality (Lemma \ref{lem2.1}) in the
last step of (\ref{R-E68}). Substituting it into (\ref{R-E67}), we
conclude that
\begin{eqnarray}
&&\|\Delta_{q}\nabla Z\|^2_{L^2_{t}(L^2)}\nonumber\\&\leq
&C2^{2q}(\|\Delta_{q}Z\|^2_{L^\infty_{T}(L^2)}+\|\Delta_{q}Z_{0}\|^2_{L^2})\nonumber\\&&+C\Big(\|(I-\mathcal{P})\Delta_{q}
Z\|^2_{L^2_{T}(L^2)}+\|(I-\mathcal{P})\Delta_{q}\nabla
Z\|^2_{L^2_{T}(L^2)}\Big)\nonumber\\&&+C\|\Delta_{q}\nabla
Z\|_{L^2_{T}(L^2)}\|\Delta_{q}\mathcal{G}\|_{L^2_{T}(L^2)}.\label{R-E69}
\end{eqnarray}
Then multiplying the factor $2^{2q(\sigma-1)}$ on both sides of
(\ref{R-E69}) gives
\begin{eqnarray}
&&2^{2q(\sigma-1)}\|\Delta_{q}\nabla
Z\|^2_{L^2_{t}(L^2)}\nonumber\\&\leq
&Cc_{q}^2(\|Z\|^2_{\widetilde{L}^\infty_{T}(B^{\sigma}_{2,1})}+\|Z_{0}\|^2_{B^{\sigma}_{2,1}})+Cc_{q}^2\Big(\|(I-\mathcal{P})
Z\|^2_{\widetilde{L}^2_{T}(B^{\sigma}_{2,1})}\Big)\nonumber\\&&+Cc_{q}^2\|\nabla
Z\|_{\widetilde{L}^2_{T}(B^{\sigma-1}_{2,1})}\|\mathcal{G}\|_{\widetilde{L}^2_{T}(B^{\sigma-1}_{2,1})}\label{R-E70}
\end{eqnarray}
with
$\|\mathcal{G}\|_{\widetilde{L}^2_{T}(B^{\sigma-1}_{2,1})}\leq\|\mathcal{G}_{1}\|_{\widetilde{L}^2_{T}(B^{\sigma-1}_{2,1})}
+\|\mathcal{G}_{2}\|_{\widetilde{L}^2_{T}(B^{\sigma-1}_{2,1})}$,
where
\begin{eqnarray}
\|\mathcal{G}_{1}\|_{\widetilde{L}^2_{T}(B^{\sigma-1}_{2,1})}\leq
C\|Z\|_{\widetilde{L}^\infty_{T}(B^{\sigma-1}_{2,1})}\|\nabla
V\|_{\widetilde{L}^2_{T}(B^{\sigma-1}_{2,1})}\label{R-E71}
\end{eqnarray}
and
\begin{eqnarray}
\|\mathcal{G}_{2}\|_{\widetilde{L}^2_{T}(B^{\sigma-1}_{2,1})}\leq\|Z\|_{\widetilde{L}^\infty_{T}(B^{\sigma-1}_{2,1})}
\|(I-\mathcal{P})V\|_{\widetilde{L}^{2}_{T}(B^{\sigma-1}_{2,1})}+C\|\tilde{A}^{0}(V)^{-1}\tilde{r}(V)\|_{\widetilde{L}^2_{T}(B^{\sigma-1}_{2,1})}.\label{R-E72}
\end{eqnarray}
From Corollaries \ref{cor6.1} and \ref{cor1.1}, we have
\begin{eqnarray}
&&\|\tilde{A}^{0}(V)^{-1}\tilde{r}(V)\|_{\widetilde{L}^2_{T}(B^{\sigma-1}_{2,1})}\nonumber\\&\leq&
C\Big(\|\tilde{A}^{0}(V)^{-1}\tilde{r}(V)\|_{L^2_{T}(L^2)}+\|\tilde{A}^{0}(V)^{-1}\tilde{r}(V)\|_{\widetilde{L}^2_{T}(\dot{B}^{\sigma-1}_{2,1})}\Big)
\nonumber\\&\leq&
C\Big(\|Z\|_{L^\infty_{T}(L^\infty)}\|(I-\mathcal{P})V\|_{L^2_{T}(L^2)}+\|\tilde{A}^{0}(V)^{-1}\|_{\widetilde{L}^\infty_{T}(\dot{B}^{\sigma-1}_{2,1})}
\|\tilde{r}(V)\|_{\widetilde{L}^2_{T}(\dot{B}^{\sigma-1}_{2,1})}\Big)\nonumber\\&\leq&C\Big(\|Z\|_{\widetilde{L}^\infty_{T}(B^{\sigma-1}_{2,1})}
\|(I-\mathcal{P})V\|_{\widetilde{L}^2_{T}(B^{\sigma}_{2,1})}+\|Z\|_{\widetilde{L}^\infty_{T}(\dot{B}^{\sigma-1}_{2,1})}
\|\nabla
V\|_{\widetilde{L}^2_{T}(\dot{B}^{\sigma-2}_{2,1})}\Big)\nonumber\\&\leq&C\|Z\|_{\widetilde{L}^\infty_{T}(B^{\sigma-1}_{2,1})}
\Big(\|(I-\mathcal{P})V\|_{\widetilde{L}^2_{T}(B^{\sigma}_{2,1})}+\|\nabla
V\|_{\widetilde{L}^2_{T}(B^{\sigma-1}_{2,1})}\Big).\label{R-E73}
\end{eqnarray}
To conclude, we combine the estimates (\ref{R-E70})-(\ref{R-E73}) to
obtain
\begin{eqnarray}
&&2^{2q(\sigma-1)}\|\Delta_{q}\nabla
Z\|^2_{L^2_{t}(L^2)}\nonumber\\&\leq
&Cc_{q}^2(\|Z\|^2_{\widetilde{L}^\infty_{T}(B^{\sigma}_{2,1})}+\|Z_{0}\|^2_{B^{\sigma}_{2,1}})+Cc_{q}^2\Big(\|(I-\mathcal{P})
Z\|^2_{\widetilde{L}^2_{T}(B^{\sigma}_{2,1})}\Big)\nonumber\\&&+Cc_{q}^2\|\nabla
Z\|_{\widetilde{L}^2_{T}(B^{\sigma-1}_{2,1})}\Big\{\|Z\|_{\widetilde{L}^\infty_{T}(B^{\sigma-1}_{2,1})}\Big(\|\nabla
V\|_{\widetilde{L}^2_{T}(B^{\sigma-1}_{2,1})}+\|(I-\mathcal{P})V\|_{\widetilde{L}^2_{T}(B^{\sigma}_{2,1})}\Big)\Big\}.\label{R-E74}
\end{eqnarray}
By employing Young's inequality, we are led to the estimate
\begin{eqnarray}
&&2^{q(\sigma-1)}\|\Delta_{q}\nabla
Z\|_{L^2_{T}(L^2)}\nonumber\\&\leq
&Cc_{q}(\|Z\|_{\widetilde{L}^\infty_{T}(B^{\sigma}_{2,1})}+\|Z_{0}\|_{B^{\sigma}_{2,1}})+Cc_{q}\|(I-\mathcal{P})
Z\|_{\widetilde{L}^2_{T}(B^{\sigma}_{2,1})}\nonumber\\&&+Cc_{q}\sqrt{\|Z\|_{\widetilde{L}^\infty_{T}(B^{\sigma}_{2,1})}}\Big(\|\nabla
V\|_{\widetilde{L}^2_{T}(B^{\sigma-1}_{2,1})}+\|(I-\mathcal{P})V\|_{\widetilde{L}^2_{T}(B^{\sigma}_{2,1})}\Big).\label{R-E75}
\end{eqnarray}
In the end, summing up (\ref{R-E75}) on $q\geq-1$, we immediately
deduce that
\begin{eqnarray}
&&\|\nabla
Z\|_{\widetilde{L}^2_{T}(B^{\sigma-1}_{2,1})}\nonumber\\&\leq
&C(\|Z_{0}\|_{B^{\sigma}_{2,1}}+E(T))+C\|(I-\mathcal{P})
Z\|_{\widetilde{L}^2_{T}(B^{\sigma}_{2,1})}+C\sqrt{E(T)}D(T),\label{R-E76}
\end{eqnarray}
which is just the desired inequality (\ref{R-E58}).\end{proof}

\noindent \textbf{\textit{The proof of Proposition \ref{prop4.1}}.}
Combining the inequalities (\ref{R-E57}) and (\ref{R-E58}), we end
up with
\begin{eqnarray}
E(T)+D(T)&\leq&
C\|V_{0}-\bar{V}\|_{B^{\sigma}_{2,1}}+C\sqrt{E(T)}D(T)\nonumber\\&\leq&
C\|V_{0}-\bar{V}\|_{B^{\sigma}_{2,1}}+C\sqrt{\delta_{1}}D(T).\label{R-E77}
\end{eqnarray}
We take $\delta_{1}>0$ suitably small such that
$C\sqrt{\delta_{1}}\leq1/2$. Then the \textit{a priori} estimate
(\ref{R-E41}) follows readily. This finishes the proof of
Proposition \ref{prop4.1}.
\hspace{74mm}$\square$\\

\noindent \textbf{\textit{The proof of Corollary \ref{cor1.2}}.}
From Remark \ref{rem2.2} an the energy inequality (\ref{R-E8}), we
get
$$\nabla \mathcal{P}U\in L^2_{t}(B^{\sigma-1}_{2,1})
\ \ \ \ \mbox{and}\ \ \ \ \nabla \mathcal{P}U_{t}\in
L^2_{t}(B^{\sigma-2}_{2,1})$$ for any $t>0$. Set
$$\mathcal{H}(t)=\|\nabla
\mathcal{P}U(\cdot,t)\|^2_{B^{\sigma-2}_{2,1}}\in L^1_{t}.$$ By
direct calculations, it is not difficult to see that
$$ \frac{d}{dt}\mathcal{H}(t)\leqslant \|\nabla
\mathcal{P}U(\cdot,t)\|^2_{B^{\sigma-2}_{2,1}}+\|\nabla
\mathcal{P}U_{t}(\cdot,t)\|^2_{B^{\sigma-2}_{2,1}}\in L^1_{t}.$$
Then apply to the embedding property in Sobolev spaces to get
$\mathcal{H}(t)\rightarrow 0$, as $t\rightarrow +\infty$.

On the other hand,  we can deduce that $\mathcal{P}U(t,x)-\bar{U}$
is bounded in $\mathcal{C}([0,\infty),B^{\sigma}_{2,1})$. Hence,
from interpolation arguments and the definition of projection
operator $\mathcal{P}$, we obtain
$$\|\nabla \mathcal{P}(U(\cdot,t)-\bar{U})\|_{B^{\sigma-1-\varepsilon}_{2,1}}\rightarrow
0\ (\varepsilon>0), \ \ \mbox{as}\ \ t\rightarrow +\infty.$$ It
further follows from Gagliardo-Nirenberg-Sobolev inequality in
\cite{E} that
$$\|\mathcal{P}U(\cdot,t)-\bar{U}\|_{B^{\sigma-1-\varepsilon}_{p,2}}\rightarrow
0 \ \ \Big(p=\frac{2d}{d-2},\ d>2\Big).$$ What remains is the
asymptotic behavior of $(I-\mathcal{P})U$, which can be dealt with
in a similar way, here, we omit the details.\hspace{114mm}$\square$

\section{Applications}\setcounter{equation}{0}\label{sec:5}
In this section, we see that the new global existence result can be
applied to some concrete partially dissipative hyperbolic equations,
for instance, the compressible Euler equations with damping, which
are given by the following form
\begin{equation}
\left\{
\begin{array}{l}
\partial_{t}\rho + \nabla\cdot(\rho\textbf{u}) = 0 , \\
\partial_{t}(\rho\textbf{u}) +\nabla\cdot(\rho\textbf{u}\otimes\textbf{u}) +
\nabla P =-\rho\textbf{u}.
\end{array} \right.\label{R-E78}
\end{equation}
Here $\rho = \rho(t, x)$ is the fluid density function of
$(t,x)\in[0,+\infty)\times\mathbb{R}^{d}$ with $d\geq1$;
$\textbf{u}=(u^1,u^2,\cdot\cdot\cdot,u^{d})^{\top}$ denotes the
fluid velocity. The pressure $P$ is related to the density by
$P(n)$, which satisfies
$$P'(\rho)>0,\ \ \ \forall\rho>0.$$ For simplicity, a usual assumption is the $\gamma$-law: $P(\rho)=A\rho^{\gamma}(\gamma\geq
1)$, where $A>0$ is some physical constant, the adiabatic exponent
$\gamma>1$ corresponds to the isentropic flow and $\gamma=1$
corresponds to the isothermal flow.

Let us consider the Cauchy problem of the compressible Euler
equations (\ref{R-E78}) subject to the initial data
\begin{equation}
(\rho,\textbf{u})(0,x)=(\rho_{0},\textbf{u}_{0}).\label{R-E79}
\end{equation}

The system (\ref{R-E78}) describes the compressible gas flow passes
a porous medium and the medium induces a friction force,
proportional to the linear momentum in the opposite direction. It is
easy to show that (\ref{R-E78}) is a strict hyperbolic system when
the smooth solutions $(\rho,\textbf{u})$ away from vacuum. For the
one-dimensional space case, the global existence of a smooth
solution with small data was obtained first by Nishida \cite{N2},
and the asymptotic behavior of the smooth solution was studied in
many papers, see e.g. the excellent survey paper by Dafermos
\cite{D1} and the book by Hsiao \cite{H}. In multi-dimension space,
as shown by \cite{STW,WY}, if the initial data are small in some
Sobolev space $H^{s}(\mathbb{R}^{d})$ with $s>1+d/2$
($s\in\mathbb{Z}$), the damping term can prevent the development of
shock waves in finite time and the Cauchy problem
(\ref{R-E78})-(\ref{R-E79}) admits a unique global classical
solution. Furthermore, it is proved that the solution in \cite{STW}
has the $L^{\infty}$ convergence rate $(1+t)^{-3/2}(d=3)$ to the
constant background state and the optimal $L^{p}(1<p\leq\infty)$
convergence rate $(1+t)^{-d/2(1-1/p)}$ in general several dimensions
\cite{WY}, respectively. For the large-time behavior of solutions
with vacuum, the reader is referred to \cite{HP2,HMP}.

Recently, in \cite{FX}, Fang and the first author studied the
existence and asymptotic behavior of classical solutions in the
framework of Besov spaces with relatively \textit{lower} regularity.
From \cite{FX}, we find that the low-frequency estimate of density
for (\ref{R-E78}) is absent. Then, we overcame the difficulty by
using Gagliardo-Nirenberg-Sobolev inequality (see, e.g., \cite{E})
to obtain a global classical solution, regretfully, the result fails
to hold in the critical Besov spaces ($\sigma=1+d/2$). In fact, we
added a little regularity ($\sigma=1+d/2+\varepsilon,
\varepsilon>0$) to ensure the global existence of classical
solutions, which is essentially different from the Euler-Poisson
equations in \cite{X}.

Now, to obtain the desired result of critical regularity, we hope
that the compressible Euler equations (\ref{R-E78}) may enter in the
class of the hyperbolic balance laws (\ref{R-E1}) of partially
dissipative structure.

To this end, set
$$U=\left(
      \begin{array}{c}
        \rho \\
        \mathbf{m} \\
      \end{array}
    \right),\ \ \
    G(U)=\left(
      \begin{array}{c}
        0 \\
        -\mathbf{m} \\
      \end{array}
    \right)\in \mathcal{M}^{\bot}
$$
with $\mathbf{m}=\rho\textbf{u}$. Allow us to abuse the notations
here. Moreover, corresponding to the orthogonal decomposition
$\mathbb{R}^{d+1}=\mathcal{M}\oplus\mathcal{M}^{\bot}$ with
$\mbox{dim}\ \mathcal{M}=1$, we write
$U=\mathcal{P}U+(I-\mathcal{P})U$ with
$$\mathcal{P}U=\left(
      \begin{array}{c}
        \rho \\
        0 \\
      \end{array}
    \right),\ \ \ (I-\mathcal{P})U=\left(
      \begin{array}{c}
        0\\
        \mathbf{m} \\
      \end{array}
    \right),$$
where $\mathcal{P}$ is the orthogonal projection onto $\mathcal{M}$.
In addition, the constant equilibrium state is
$(\bar{\rho},0)^{\top}\in \mathcal{E}$.

Let us introduce the energy function
$$\eta(U):=\frac{|\mathbf{m}|^2}{2\rho}+h(\rho)\ \ \ \mbox{with}\ \
\ h'(\rho)=\int^{\rho}_{1}\frac{P'(s)}{s}ds.$$ Define
$$W=\left(
      \begin{array}{c}
        W_{1} \\
        W_{2} \\
      \end{array}
    \right):=\nabla\eta(U)=\left(
                             \begin{array}{c}
                               -\frac{|\mathbf{m}|^2}{2\rho^2}+h'(\rho) \\
                               \mathbf{m}/\rho \\
                             \end{array}
                           \right).
$$
Clearly, the mapping $U\rightarrow W$ is a diffeomorphism from
$\mathcal{O}_{U}:=\mathbb{R}^{+}\times \mathbb{R}^{d}$ onto its
range $\mathcal{O}_{W}$. Then for classical solutions
$(\rho,\textbf{u})$ away from vacuum, (\ref{R-E78}) is equivalent to
the following system
\begin{equation}
A^{0}(W)W_{t}+\sum_{j=1}^{d}A^{j}(W)W_{x_{j}}=H(W) \label{R-E80}
\end{equation}
with $$A^{0}(W)=\left(
          \begin{array}{cc}
            1 & W_{2}^{\top} \\
            W_{2} & W_{2}\otimes W_{2}+P'(\rho)I_{d} \\
          \end{array}
        \right),$$

$$A^{j}(W)=\left(
          \begin{array}{cc}
            W_{2j} & W_{2}^{\top}W_{2j}+P'(\rho)e_{j}^{\top} \\
            W_{2}W_{2j}+P'(\rho)e_{j} & W_{2j}(W_{2}\otimes W_{2}+P'(\rho)I_{d})+P'(\rho)(W_{2}\otimes e_{j}+e_{j}\otimes W_{2}) \\
          \end{array}
        \right),$$

$$H(W)=G(U(W))=\left(
                 \begin{array}{c}
                   0 \\
                   -P'(\rho)W_{2} \\
                 \end{array}
               \right)
,$$ where $I_{d}$ stands for the $d\times d$ unit matrix, and
$e_{j}$ is $d$-dimensional vector where the $j$th component is one,
others are zero. It follows from the definition of variable $W$ that
$h'(\rho)=W_{1}+|W_{2}|^2/2$, so $P'(\rho)$ can be viewed as a
function of $W$, since $\rho$ is the function of
$W_{1}+|W_{2}|^2/2$, i.e. of $W$.

It is not difficult to see that $A^{0}(W)$ is symmetric and positive
definite, and the matrices $A^{j}(W)(j=1,\cdot\cdot\cdot,d)$ are
symmetric. Obviously, $H(W)=0$ holds if and only if
$W\in\mathcal{M}$. In addition,
$$L(W)=-\nabla H(W)=\left(
          \begin{array}{cc}
            0 & 0 \\
            P''(\rho)\frac{\partial\rho}{\partial W_{1}}W_{2} &  P''(\rho)\frac{\partial\rho}{\partial W_{2}}\otimes W_{2}+P'(\rho)I_{d} \\
          \end{array}
        \right).
$$
For $W\in\mathcal{M}$, i.e. $W_{2}=0$, the matrix $L(W)$ is
symmetric and nonnegative definite, and its null space coincides
with $\mathcal{M}$. Hence, the system (\ref{R-E80}) is symmetric
dissipative in the sense of Definition \ref{defn1.2}.

Set $\tilde{\eta}(W)=\rho h'(\rho)-h(\rho)$, which can be viewed as
the function of $W$. Then,
$$D_{W}\tilde{\eta}(W)=\left(
          \begin{array}{cc}
            \frac{\partial\rho}{\partial W_{1}}h'(\rho)+\rho-h'(\rho) \frac{\partial\rho}{\partial W_{1}} \\
            h'(\rho)\frac{\partial\rho}{\partial W_{2}}+\rho W_{2}-h'(\rho) \frac{\partial\rho}{\partial W_{2}} \\
          \end{array}
        \right)=\left(
          \begin{array}{cc}
            \rho \\
            \rho\cdot\frac{\mathbf{m}}{\rho} \\
          \end{array}
        \right)=U,
$$
so it follows from the proof of Theorem \ref{thm1.1} in \cite{KY}
that $\langle U,W(U)\rangle-\tilde{\eta}(W)$ is an entropy function
in the sense of Definition \ref{defn1.1}. By a straightforward
computation, we find
$$\langle U,W(U)\rangle-\tilde{\eta}(W)=(\rho,\mathbf{m})\left(
                             \begin{array}{c}
                               -\frac{|\mathbf{m}|^2}{2\rho^2}+h'(\rho) \\
                               \mathbf{m}/\rho \\
                             \end{array}
                           \right)-(\rho h'(\rho)-h(\rho))=\eta(U).$$
Hence, the energy function $\eta(U)$ is just the desired entropy.
Furthermore, the associated entropy flux is
$$q(U)=\Big(\frac{|\mathbf{m}|^2}{2\rho^2}+\rho h'(\rho)\Big)\frac{\mathbf{m}}{\rho}.$$

Next, we show the linearized equations of the dissipative hyperbolic
system (\ref{R-E80}) satisfies the stability condition, see
Definition \ref{defn1.4}. Set $\bar{U}=(\bar{\rho},0)^{\top}\in
\mathcal{E}$. The corresponding equilibrium state for the variable
$W$ is $\bar{W}=(h'(\bar{\rho}),0)^{\top}\in \mathcal{M}$. The
linearized equations near $\bar{W}=(h'(\bar{\rho}),0)^{\top}$ reads
\begin{equation}
A^{0}W_{t}+\sum_{j=1}^{d}A^{j}W_{x_{j}}+L=0, \label{R-E81}
\end{equation}
where
$$A^{0}=A^{0}(\bar{W})=\left(
          \begin{array}{cc}
            1 & 0 \\
            0 & P'(\bar{\rho})I_{d} \\
          \end{array}
        \right),$$

$$A^{j}=A^{j}(\bar{W})=\left(
          \begin{array}{cc}
            0 & P'(\bar{\rho})e_{j}^{\top} \\
            P'(\bar{\rho})e_{j} & O \\
          \end{array}
        \right),$$

$$L=L(\bar{W})=\left(
                 \begin{array}{cc}
                   0 & 0\\
                  0& P'(\bar{\rho})I_{d} \\
                 \end{array}
               \right).$$
\textbf{Claim}: Let
$\varphi=(\widetilde{W}_{1},\widetilde{W}_{2})^{\top}\in
\mathbb{R}^{d+1},\ \ \lambda\in \mathbb{R}$ and $\omega \in
\mathbb{S}^{d-1}$. If $L\varphi=0$ and $\lambda
A^{0}\varphi+A(\omega)\varphi=0$, then $\varphi=0$.

Indeed, the condition $L\varphi=0$ gives $\widetilde{W}_{2}=0$. Then
$\varphi=(\widetilde{W}_{1},0)^{\top}$. For this $\varphi$, we
suppose that $\lambda A^{0}\varphi+A(\omega)\varphi=0$, which
implies $\lambda\widetilde{W}_{1}=0$ and
$P'(\bar{\rho})\widetilde{W}_{1}\omega=0$. This shows that
$\widetilde{W}_{1}=0$ and hence $\varphi=0$.

Over all, the compressible Euler equations with damping are included
in the class of dissipative hyperbolic equations and satisfy the
stability condition at $\bar{W}$. Therefore, from Theorems
\ref{thm1.4}-\ref{thm1.5}, the Cauchy problem
(\ref{R-E78})-(\ref{R-E79}) admits the local classical solutions for
general data in critical spaces.

\begin{thm} \label{thm5.1} Let $\bar{\rho}>0$ be a constant reference density.
Suppose that \ $\rho_{0}-\bar{\rho}$ and $\mathbf{m}_{0}\in
B^{\sigma}_{2,1}$ with $\rho_{0}>0$, then there exists a time
$T_{1}>0$ such that
\begin{itemize}
\item[(i)] Existence:  the
Cauchy problem (\ref{R-E78})-(\ref{R-E79}) has a unique solution
$(\rho,\mathbf{m})\in \mathcal{C}^{1}([0,T_{1}] \times
\mathbb{R}^{d})$ with $\rho>0$ for all $t\in [0,T_{1}]$ and
$$(\rho-\bar{\rho},\mathbf{m})\in \widetilde{\mathcal{C}}_{T_{1}}(B^{\sigma}_{2,1})
\cap\widetilde{\mathcal{C}}^{1}_{T_{0}}(B^{\sigma-1}_{2,1}).
$$

\item[(ii)] Blow-up criterion: there exists a constant $C_{0}>0$
such that the maximal time $T^{*}$ of existence of such a solution
can be bounded from below by
$T^{*}\geq\frac{C_{0}}{\|(\rho_{0}-\bar{\rho},\
\mathbf{m}_{0})\|_{B^{\sigma}_{2,1}}}.$ Moreover, if $T^{*}$ is
finite, then
$$\limsup_{t\rightarrow T^{*}}\|(\rho-\bar{\rho},\
\mathbf{m})\|_{B^{\sigma}_{2,1}}=\infty$$ if and only if
$$\int^{T^{*}}_{0}\|(\nabla \rho,\nabla\mathbf{m})\|_{L^{\infty}}dt=\infty.$$
\end{itemize}
\end{thm}

Under the assumption of small initial data, we further obtain
global-in-time classical solutions to (\ref{R-E78})-(\ref{R-E79}) in
critical spaces.

\begin{thm}\label{thm5.2}
Let $\bar{\rho}>0$ be a constant reference density. Suppose that \
$\rho_{0}-\bar{\rho}$ and $\mathbf{m}_{0}\in B^{\sigma}_{2,1}$,
there exists a positive constant $\tilde{\delta}_{0}$ such that if
$$\|(\rho_{0}-\bar{\rho},\mathbf{m}_{0})\|_{B^{\sigma}_{2,1}}\leq
\tilde{\delta}_{0}$$ with $\mathbf{m}_{0}=\rho_{0}\mathbf{u}_{0}$,
then the Cauchy problem (\ref{R-E78})-(\ref{R-E79}) has a unique
global solution $(\rho,\mathbf{m})$ satisfying
\begin{eqnarray*}
(\rho,\mathbf{m})\in \mathcal{C}^{1}(\mathbb{R}^{+}\times
\mathbb{R}^{d})
\end{eqnarray*}and
\begin{eqnarray*}
(\rho-\bar{\rho},\mathbf{m}) \in
\widetilde{\mathcal{C}}(B^{\sigma}_{2,1})\cap
\widetilde{\mathcal{C}}^1(B^{\sigma-1}_{2,1}).
\end{eqnarray*}
Furthermore, it holds that\begin{eqnarray*}
&&\|(\rho-\bar{\rho},\mathbf{m})\|_{\widetilde{L}^\infty(B^{\sigma}_{2,1})}
+\tilde{\mu}_{0}\Big(\|\mathbf{m}\|_{\widetilde{L}^2(B^{\sigma}_{2,1})}
+\|(\nabla\rho,\nabla\mathbf{m})\|_{\widetilde{L}^2(B^{\sigma-1}_{2,1})}\Big)
\nonumber\\&\leq& \tilde{C}_{0}\|(\rho_{0}-\bar{\rho},
\mathbf{m}_{0})\|_{B^{\sigma}_{2,1}},
\end{eqnarray*}
where $\tilde{C}_{0},\tilde{\mu}_{0}$ are some positive constants.
\end{thm}

\begin{rem}
In comparison with the previous works \cite{STW,FX}, Theorem
\ref{thm5.1} indicates the critical regularity of global classical
solutions. In addition, the result is valid for the \textit{general}
pressure in \textit{arbitrary} space dimensions (Due to the
techniques used, the papers \cite{STW,FX} were denoted to the system
(\ref{R-E78})-(\ref{R-E79}) with \textit{$\gamma$-law} pressure in
at least \textit{three} space dimensions).
\end{rem}

From Corollary \ref{cor1.2}, we also get the corresponding
large-time behavior of $(\rho,\mathbf{m})$.
\begin{cor}
Let $(\rho,\mathbf{m})$ be the solution in Theorem \ref{thm5.1}.
Then
$$\|\nabla\rho(\cdot,t)\|_{B^{\sigma-1-\varepsilon}_{2,1}}\rightarrow
0,$$

$$\|\rho(\cdot,t)-\bar{\rho}\|_{B^{\sigma-1-\varepsilon}_{p,2}}\rightarrow 0 \ \
\Big(p=\frac{2d}{d-2},\ d>2\Big),$$ and
$$\|\mathbf{m}(\cdot,t)\|_{B^{\sigma-\varepsilon}_{2,1}}\rightarrow
0$$ for any  $\varepsilon>0$, as $t\rightarrow +\infty$.
\end{cor}

\section{Appendix}\setcounter{equation}{0}\label{sec:6}
In this section, we first investigate the connection between the
homogeneous Chemin-Lerner's spaces and inhomogeneous Chemin-Lerner's
spaces. Precisely, we have
\begin{prop}\label{prop6.1}
Let $s>0,\ 1\leq \theta, p,r\leq\infty$. When $\theta\geq r$, it
holds that
\begin{eqnarray}
L^{\theta}_{T}(L^{p})\cap
\widetilde{L}^{\theta}_{T}(\dot{B}^{s}_{p,r})=\widetilde{L}^{\theta}_{T}(B^{s}_{p,r})\label{R-E82}
\end{eqnarray}
for any $T>0$.
\end{prop}
\begin{proof} Without loss of generality, we deal with $r<\infty$ only.
In order to prove (\ref{R-E82}), we first show for all $s\in
\mathbb{R}$
$$L^{\theta}_{T}(L^{p})\cap
\widetilde{L}^{\theta}_{T}(\dot{B}^{s}_{p,r})\subset\widetilde{L}^{\theta}_{T}(B^{s}_{p,r}).$$
Indeed, if $f\in L^{\theta}_{T}(L^{p})$, then
$$\|\Delta_{-1}f\|_{L^{\theta}_{T}(L^{p})}=\|\Psi\ast f\|_{L^{\theta}_{T}(L^{p})}\leq C\|f\|_{L^{\theta}_{T}(L^{p})}.$$
On the other hand, if
$f\in\widetilde{L}^{\theta}_{T}(\dot{B}^{s}_{p,r})$, then
$$\Big\{\sum_{q=-\infty}^{\infty}\Big(2^{qs}\|\dot{\Delta}_{q}f\|_{L^{\theta}_{T}(L^{p})}\Big)^{r}\Big\}^{1/r}
=\Big\{\sum_{q=-\infty}^{\infty}\Big(2^{qs}\|\Phi_{q}\ast
f\|_{L^{\theta}_{T}(L^{p})}\Big)^{r}\Big\}^{1/r}<+\infty.$$ Of
course,
$$\Big\{\sum_{q=0}^{\infty}\Big(2^{qs}\|\Phi_{q}\ast
f\|_{L^{\theta}_{T}(L^{p})}\Big)^{r}\Big\}^{1/r}=\Big\{\sum_{q=0}^{\infty}\Big(2^{qs}\|\Delta_{q}
f\|_{L^{\theta}_{T}(L^{p})}\Big)^{r}\Big\}^{1/r}<+\infty.$$ Hence,
we arrive at
$$\Big\{\sum_{q=-1}^{\infty}\Big(2^{qs}\|\Delta_{q}
f\|_{L^{\theta}_{T}(L^{p})}\Big)^{r}\Big\}^{1/r}<+\infty.$$ That is,
$f\in\widetilde{L}^{\theta}_{T}(B^{s}_{p,r}).$

Conversely, thanks to Remark \ref{rem2.2} ($\theta\geq r$) and
$s>0$, we have the following imbeddings
$$\widetilde{L}^{\theta}_{T}(B^{s}_{p,r})\hookrightarrow L^{\theta}_{T}(B^{s}_{p,r})
\hookrightarrow L^{\theta}_{T}(B^{0}_{p,1})\hookrightarrow
L^{\theta}_{T}(L^{p}), $$ so if $f\in
\widetilde{L}^{\theta}_{T}(B^{s}_{p,r})$ then $f\in
L^{\theta}_{T}(L^{p})$.

In addition, for $q<0$, we obtain
$$\|\dot{\Delta}_{q}
f\|_{L^{\theta}_{T}(L^{p})}=\cases{\|\Phi_{q}\ast\Psi\ast
f\|_{L^{\theta}_{T}(L^{p})}\leq C\|\Delta_{-1}
f\|_{L^{\theta}_{T}(L^{p})},\ \ \ q<-1,
\cr\|\dot{\Delta}_{-1}f\|_{L^{\theta}_{T}(L^{p})}\leq
C\|f\|_{L^{\theta}_{T}(L^{p})}, \ \ \ q=-1.}
$$ Hence, when $s>0$, it is not difficult to get
\begin{eqnarray*}\|f\|_{\widetilde{L}^{\theta}_{T}(\dot{B}^{s}_{p,r})}
&=&\Big\{\sum_{q=-\infty}^{\infty}\Big(2^{qs}\|\dot{\Delta}_{q}f\|_{L^{\theta}_{T}(L^{p})}\Big)^{r}\Big\}^{1/r}
\\&\leq&
\Big\{\sum_{q<0}\Big(2^{qs}\|\dot{\Delta}_{q}f\|_{L^{\theta}_{T}(L^{p})}\Big)^{r}\Big\}^{1/r}
+\Big\{\sum_{q\geq0}\Big(2^{qs}\|\dot{\Delta}_{q}f\|_{L^{\theta}_{T}(L^{p})}\Big)^{r}\Big\}^{1/r}\\&\leq&
C\sum_{q<-1}2^{qs}\|\dot{\Delta}_{q}f\|_{L^{\theta}_{T}(L^{p})}+C2^{-s}\|\dot{\Delta}_{-1}f\|_{L^{\theta}_{T}(L^{p})}
+\Big\{\sum_{q\geq0}\Big(2^{qs}\|\dot{\Delta}_{q}f\|_{L^{\theta}_{T}(L^{p})}\Big)^{r}\Big\}^{1/r}\\&\leq&
C\|\Delta_{-1}f\|_{L^{\theta}_{T}(L^{p})}+\Big\{\sum_{q\geq0}\Big(2^{qs}\|\Delta_{q}f\|_{L^{\theta}_{T}(L^{p})}\Big)^{r}\Big\}^{1/r}
+C\|f\|_{L^{\theta}_{T}(L^{p})}
\\&\leq& C\|f\|_{\widetilde{L}^{\theta}_{T}(B^{s}_{p,r})}.
\end{eqnarray*}
Therefore,  $f\in
L^{\theta}_{T}(L^{p})\cap\widetilde{L}^{\theta}_{T}(\dot{B}^{s}_{p,r})$
if $f\in \widetilde{L}^{\theta}_{T}(B^{s}_{p,r})$.
\end{proof}

\begin{cor}\label{cor6.1}
Let the assumptions of Proposition \ref{prop6.1} be fulfilled. Then
$$\|f\|_{\widetilde{L}^{\theta}_{T}(B^{s}_{p,r})}\approx\|f\|_{L^{\theta}_{T}(L^{p})}
+\|f\|_{\widetilde{L}^{\theta}_{T}(\dot{B}^{s}_{p,r})}.$$
\end{cor}

In what follows, we are concerned with the existence result for the
linear problem (\ref{R-E10})-(\ref{R-E11}) in more general Besov
spaces, which is used to establish the local existence in
$B^{\sigma}_{2,1}$ for the quasilinear symmetric system (\ref{R-E5})
with (\ref{R-E9}).
\begin{prop}\label{prop6.2}
Let $T>0$ and let $r\in[1,\infty), s>0$ and $V_{0}-\bar{V}\in
B^{s}_{2,r}$. Assume that
\begin{eqnarray*}&&V-\bar{V}\in \left\{
\begin{array}{l}
 \mathcal{C}_{T}(B^{s}_{2,r})\cap \mathcal{C}^{1}_{T}(B^{s-1}_{2,r})\ \ \mbox{if}\ \ s>1+d/2, \mbox{or}\ s=1+d/2 \ \mbox{and}\ r=1;\\
 \mathcal{C}_{T}(B^{s+\varepsilon}_{2,\infty})\cap \mathcal{C}^{1}_{T}(B^{s-1+\varepsilon}_{2,\infty})\ \mbox{for}\ \varepsilon>0\ \mbox{if}\ \ s=1+d/2 \ \mbox{and}\ r>1;\\
 \mathcal{C}_{T}(B^{1+d/2}_{2,\infty}\cap W^{1,\infty})\cap\mathcal{C}^{1}_{T}(B^{d/2}_{2,\infty}\cap L^{\infty}) \ \ \mbox{if}\ \ 0<s<1+d/2;\\
\end{array} \right.
\end{eqnarray*}
\begin{eqnarray*}
V(t,x)\in\mathcal{O}_{1}\ \ \ \mbox{for  any}\ \ (t,x)\in Q_{T},
\end{eqnarray*}
where $\mathcal{O}_{1}$ is a bounded open convex set in
$\mathbb{R}^{N}$ satisfying $\bar{\mathcal{O}}_{1}\subset
\mathcal{O}_{V}.$ Then the system (\ref{R-E10})-(\ref{R-E11}) has a
unique solution $\hat{V}$ belongs to
$$\hat{V}-\bar{V}\in\widetilde{\mathcal{C}}_{T}(B^{s}_{2,r}) \cap\widetilde{\mathcal{C}}^{1}_{T}(B^{s-1}_{2,r})$$
and satisfies
\begin{eqnarray}
\|\hat{V}-\bar{V}\|_{\widetilde{L}^{\infty}_{T}(B^{s}_{2,r})}\leq\|V_{0}-\bar{V}\|_{B^{s}_{2,r}}e^{C\int^{T}_{0}(a_{1}(t)+a_{2}(t))dt},\label{R-E83}
\end{eqnarray}
where $$a_{1}(t)=\|V_{t}(t,\cdot)\|_{L^\infty},$$
\begin{eqnarray*}&a_{2}(t)=\left\{
\begin{array}{l}
\|V(t,\cdot)-\bar{V}\|_{B^{s}_{2,r}} \ \mbox{if}\ \ s>1+d/2,\ \mbox{or}\ \ s=1+d/2 \ \mbox{and}\ r=1;\\
\|V(t,\cdot)-\bar{V}\|_{B^{s+\varepsilon}_{2,\infty}} \ \mbox{for}\ \ \varepsilon>0\ \ \mbox{if}\ \ s=1+d/2 \ \mbox{and}\ r>1;\\
\|V(t,\cdot)-\bar{V}\|_{B^{1+d/2}_{2,\infty}\cap W^{1,\infty}}\ \
\mbox{if}\ \ 0<s<1+d/2.
\end{array} \right.
\end{eqnarray*}
\end{prop}

\begin{proof}
The energy inequality (\ref{R-E83}) can follow from the proof of
Lemma \ref{lem3.1} at a similar way, however, we should point out
the estimates of commutator
$\|[\Delta_{q},\tilde{A}^{0}(V)^{-1}\tilde{A}^{j}(V)]\hat{Z}_{x_{j}}\|_{L^2}$
for general indexes $s,r$. Precisely, from \cite{BCD}, we have
\begin{lem}\label{lem6.1} For all $t\in[0,T]$ and $s>0$, it holds that
\begin{eqnarray*}
&&2^{qs}\|[\Delta_{q},\tilde{A}^{0}(V)^{-1}\tilde{A}^{j}(V)]\hat{Z}_{x_{j}}\|_{L^2}\nonumber\\
&\leq& \left\{
\begin{array}{l}
Cc_{q}\|\nabla(\tilde{A}^{0}(V)^{-1}\tilde{A}^{j}(V))\|_{B^{s-1}_{2,r}}\|\hat{Z}\|_{B^{s}_{2,r}},\
\mbox{if} \ s>1+d/2,\ \mbox{or}\ \ s=1+d/2 \ \mbox{and}\ r=1;\\
Cc_{q}\|\nabla(\tilde{A}^{0}(V)^{-1}\tilde{A}^{j}(V))\|_{B^{s-1+\varepsilon}_{2,\infty}}\|\hat{Z}\|_{B^{s}_{2,r}},\
\mbox{for}\ \varepsilon>0\ \mbox{if}\ \ s=1+d/2 \ \mbox{and}\
r>1;\\
Cc_{q}\|\nabla(\tilde{A}^{0}(V)^{-1}\tilde{A}^{j}(V))\|_{B^{d/2}_{2,\infty}\cap
L^{\infty}}\|\hat{Z}\|_{B^{s}_{2,r}},\ \mbox{if}\ \ 0<s<1+d/2,
\end{array} \right.
\end{eqnarray*}
where $\|c_{q}(t)\|_{\ell^1}\leq1$, for all $t\in[0,T]$.
\end{lem}
Note that these facts, similar to the steps
(\ref{R-E16})-(\ref{R-E23}), we readily deduce that
\begin{eqnarray}
\|\Delta_{q}\hat{Z}\|_{L^{\infty}_{t}(L^2)}&\leq&\|\Delta_{q}\hat{Z}_{0}\|_{L^2}
+C\int^{t}_{0}c_{q}(\tau)2^{-qs}a_{2}(\tau)\|\hat{Z}\|_{B^{s}_{2,r}}d\tau\nonumber
\\&&+C\int^{t}_{0}\|\mathrm{div}\mathbb{A}(V)\|_{L^\infty}\|\Delta_{q}\hat{Z}\|_{L^2}d\tau. \label{R-E84}
\end{eqnarray}
Then we multiply both sides by $2^{qs}$ and take the $\ell^{r}$ norm
to obtain
\begin{eqnarray}
\|\hat{Z}\|_{\widetilde{L}^{\infty}_{T}(B^{s}_{2,r})}&\leq&\|\hat{Z}_{0}\|_{B^{s}_{2,r}}
+C\int^{T}_{0}(a_{1}(t)+a_{2}(t))\|\hat{Z}\|_{\widetilde{L}^{\infty}_{t}(B^{s}_{2,r})}dt\nonumber.\label{R-E85}
\end{eqnarray}
Applying Gronwall's inequality leads to the inequality (\ref{R-E83})
directly.

To show the existence of solution $\hat{V}(t,x)$, we use the
classical Friedrichs' regularization method, which was used in
\cite{CH2} for example. More precisely, we consider the approximate
system for $\hat{Z}_{k}=\hat{V}_{k}-\bar{V}$:
\begin{equation}
\partial_{t}\hat{Z}_{k}+\sum_{j=1}^{d}\tilde{A}^{0}(V)^{-1}\mathbb{P}_{k}\Big(\tilde{A}^{j}(V)\partial_{x_{j}}\hat{Z}_{k}\Big)=0,\label{R-E86}
\end{equation}
with
\begin{equation}
\hat{Z}_{k}|_{t=0}=\mathbb{P}_{k}\hat{Z}_{0},\label{R-E87}
\end{equation}
where $\hat{Z}_{0}=\hat{V}_{0}-\bar{V}$ and $\mathbb{P}_{k}$ is the
cut-off operator on $L^2(\mathbb{R}^{d})$ defined by
$$\mathbb{P}_{k}f:=\mathcal{F}^{-1}(\mathbf{1}_{B(0,k)}\mathcal{F}f).$$
Denote the space
$$L^2_{k}:=\{f\in L^2(\mathbb{R}^{d}):\mathrm{supp}\mathcal{F}f\subset B(0,k)\},$$
where $B(0,k)$ is the ball with center $0$ and radius $k$.

From Lemma \ref{lem2.1}, we can see that the operator
$\partial_{x_{j}}$ is continuous on $L^2_{k}$. Furthermore, it turns
out that the linear operator
$$\hat{Z}\mapsto\sum_{j=1}^{d}\tilde{A}^{0}(V)^{-1}\mathbb{P}_{k}\Big(\tilde{A}^{j}(V)\partial_{x_{j}}\hat{Z}\Big)$$
is also continuous on $L^2_{k}$, since the functions
$\tilde{A}^{0}(V)^{-1}$ and $\tilde{A}^{j}(V)$ are both bounded in
$[0,T]\times \mathbb{R}^{d}$. Thus, the approximate system
(\ref{R-E86}) appears to a linear system of ordinary differential
equations in $L^2_{k}$, which implies the existence of a unique
function $\hat{V}_{k}(t,x)$ such that
$\hat{Z}_{k}(t,x)=\hat{V}_{k}(t,x)-\bar{V}\in
\mathcal{C}^{1}([0,T],L^2_{k})$ is the solution of
(\ref{R-E86})-(\ref{R-E87}). Furthermore, it follows from the
spectral localization that $\hat{Z}_{k}(t,x)\in
\mathcal{C}^{1}([0,T], B^{\alpha}_{2,r})$ for any $\alpha\in
\mathbb{R}$.

Using the facts that the operator $\mathbb{P}_{k}$ is self-adjoint
on $L^2$ and $\mathbb{P}_{k}\hat{Z}_{k}=\hat{Z}_{k}$, we proceed
exactly as in the proof of the inequality (\ref{R-E83}) and obtain
\begin{eqnarray}
\sup_{t\in[0,T]}\|\hat{Z}_{k}(t)\|_{B^{s}_{2,r}}&\leq&\|\mathbb{P}_{k}\hat{Z}_{0}\|_{B^{s}_{2,r}}e^{C\int^{T}_{0}(a_{1}(t)+a_{2}(t)+a^2_{2}(t))dt}
\nonumber\\&\leq&C\|V_{0}-\bar{V}\|_{B^{s}_{2,r}}e^{C\int^{T}_{0}(a_{1}(t)+a_{2}(t)+a^2_{2}(t))dt}\nonumber\\&\leq&C.\label{R-E88}
\end{eqnarray}
Here and below, the constant $C>0$ independent of $k$. Furthermore,
it follows from  (\ref{R-E86}) and (\ref{R-E88}) that
\begin{eqnarray}
\sup_{t\in[0,T]}\|\partial_{t}\hat{Z}_{k}\|_{B^{s-1}_{2,r}}\leq C.
\label{R-E89}
\end{eqnarray}

Therefore, we deduce that the approximative solution sequence
$\{\hat{Z}_{k}=\hat{V}_{k}-\bar{V}\}$ to (\ref{R-E86})-(\ref{R-E87})
is uniformly bounded in $\mathcal{C}([0,T],
B^{s}_{2,r})\cap\mathcal{C}^{1}([0,T], B^{s-1}_{2,r})$. Moreover, it
weak $^{\star}$-converges (up to a subsequence) to some function
$\hat{V}$ such that $\hat{V}-\bar{V}\in
L^{\infty}([0,T],B^{s}_{2,r})$ in terms of the Banach-Alaoglu
Theorem (see~\cite{T}, Remark 2, p.180). Since
$\{\partial_{t}\hat{V}_{k}\} $ is also uniformly bounded in
$\mathcal{C}([0,T],B^{s-1}_{2,r} )$(it weak $^{\star}$-converges to
$\hat{V}_{t}$ in $L^{\infty}([0,T],B^{s-1}_{2,r})$, then
$\{\hat{V}_{k}-\bar{V}\}$ is uniformly bounded in
Lip$([0,T],B^{s-1}_{2,r})$, hence uniformly equicontinuous on
$[0,T]$ with the norm in $B^{s-1}_{2,r}$. From Proposition
\ref{prop2.1}, Ascoli-Arzela theorem and Cantor diagonal process, we
arrive at
\begin{eqnarray}\phi_{j}(\hat{V}_{k}-\bar{V})\rightarrow \phi_{j} (\hat{V}-\bar{V})\ \ \ \mbox{in}\ \
\mathcal{C}([0,T],B^{s-1}_{2,r})\label{R-E90}
\end{eqnarray} as $k\rightarrow\infty$, for
$\phi_{j}\in C_{c}^{\infty}$ which is supported in the ball
$B(0;j+1)$ and equal to $1$ on $B(0;j)$. The property of strong
convergence and
\begin{eqnarray}\lim_{n\rightarrow\infty}\mathbb{P}_{k}(V_{0}-\bar{V})=V_{0}-\bar{V}\ \ \mbox{in}\ \ \ B^{s}_{2,r}\label{R-E91}
\end{eqnarray}
enable us to pass to the limit in (\ref{R-E86})-(\ref{R-E87}) and
$\hat{V}$ is the solution of (\ref{R-E10})-(\ref{R-E11}) in the
sense of distribution. Next, we check the solution $\hat{V}$ has the
required regularity. Indeed,
$\hat{V}-\bar{V}\in\mathcal{C}([0,T],B^{s-1}_{2,r})$. On the other
hand, since $\hat{V}-\bar{V}\in L^{\infty}([0,T],B^{s}_{2,r})$, we
have
\begin{eqnarray}\|2^{qs}\|\Delta_{q}(\hat{V}-\bar{V})\|_{L^2}\|_{\ell^{r}}<+\infty.\label{R-E92}
\end{eqnarray}
The inequality (\ref{R-E92}) implies that there exists an integer
$q_{0}$ such that
\begin{eqnarray}\Big\{\sum_{q\geq q_{0}}\Big(2^{qs}\|\Delta_{q}(\hat{V}-\bar{V})\|_{L^2}\Big)^{r}\Big\}^{1/r}<\frac{\varepsilon}{4}\label{R-E93}
\end{eqnarray}
for any positive constant $\varepsilon$. Then, we have
\begin{eqnarray}
&&\|\hat{V}(t)-\hat{V}(t')\|_{B^{s}_{2,r}}\nonumber\\&\leq&\Big\{\sum_{q<
q_{0}}\Big(2^{qs}\|\Delta_{q}(\hat{V}(t)-\hat{V}(t'))\|_{L^2}\Big)^{r}\Big\}^{1/r}+2\Big\{\sum_{q\geq
q_{0}}\Big(2^{qs}\|\Delta_{q}(\hat{V}(t)-\bar{V})\|_{L^2}\Big)^{r}\Big\}^{1/r}\nonumber\\&\leq&
C\sum_{q<q_{0}}2^{qs}\|\Delta_{q}(\hat{V}(t)-\hat{V}(t'))\|_{L^2}+\frac{\varepsilon}{2}\nonumber\\&\leq&C2^{q_{0}s}
\|\hat{V}(t)-\hat{V}(t')\|_{L^2}+\frac{\varepsilon}{2}.\label{R-E94}
\end{eqnarray}
Since
$\hat{V}-\bar{V}\in\mathcal{C}([0,T],B^{s-1}_{2,r})\hookrightarrow\mathcal{C}([0,T],L^2)$
for the case of $s>1$, the first term on the right-hand side of
(\ref{R-E94}) tends to zero where $t'$ goes to $t$. This implies
that $\hat{V}-\bar{V}$ is continuous in time with values in
$B^{s}_{2,r}$. Using the fact that $\hat{V}$ is a solution of
(\ref{R-E10})-(\ref{R-E11}), we further conclude that
$\hat{V}-\bar{V}\in\mathcal{C}([0,T], B^{s}_{2,r})
\cap\mathcal{C}^{1}([0,T],B^{s-1}_{2,r})$. In the case where
$0<s\leq1$, we consider the regularized system by smoothing out the
initial data:
\begin{equation}
\left\{
\begin{array}{l}
\tilde{A}^{0}(V)\partial_{t}\hat{Z}_{k}+\sum_{j=1}^{d}\tilde{A}^{j}(V)\partial_{x_{j}}\hat{Z}_{k}=0,\\
\hat{Z}_{k}|_{t=0}=\mathbb{P}_{k}\hat{Z}_{0}.\end{array} \right.
\label{R-E95}
\end{equation}
Thanks to the above result for the case of $s>1$, the solution
$\hat{Z}_{k}=\hat{V}_{k}-\bar{V}$ of (\ref{R-E95}) is well defined
on $[0,T]$ and belongs to $\mathcal{C}([0,T], B^{\alpha}_{2,r})
\cap\mathcal{C}^{1}([0,T],B^{\alpha-1}_{2,r})$ for any $\alpha>1$.
Furthermore, the function
$\delta\hat{Z}_{k}:=\hat{Z}_{k+1}-\hat{Z}_{k}$ satisfies
\begin{equation}
\left\{
\begin{array}{l}
\tilde{A}^{0}(V)\partial_{t}\delta\hat{Z}_{k}+\sum_{j=1}^{d}\tilde{A}^{j}(V)\partial_{x_{j}}\delta\hat{Z}_{k}=0,\\
\delta\hat{Z}_{k}|_{t=0}=(\mathbb{P}_{k+1}-\mathbb{P}_{k})\hat{Z}_{0}.\end{array}
\right. \label{R-E96}
\end{equation}
Similar to (\ref{R-E83}), we have
\begin{eqnarray}
\sup_{t\in[0,T]}\|\delta\hat{Z}_{k}(t)\|_{B^{s}_{2,r}}\leq\|(\mathbb{P}_{k+1}-\mathbb{P}_{k})\hat{Z}_{0}\|_{B^{s}_{2,r}}e^{C\int^{T}_{0}(a_{1}(t)+a_{2}(t))dt}.\label{R-E97}
\end{eqnarray}
As $\hat{Z}_{0}=\hat{V}_{0}-\bar{V}$ belongs to
$B^{s}_{2,r}(0<s\leq1)$, the sequence
$(\mathbb{P}_{k}\hat{Z}_{0})_{k\in\mathbb{N}}$ converges to
$\hat{Z}_{0}$. Thus, it follows from (\ref{R-E97}) that the sequence
$(\hat{Z}_{k})_{k\in\mathbb{N}}$ is Cauchy in $\mathcal{C}([0,T],
B^{s}_{2,r})$ and converges to some
$\hat{Z}=\hat{V}-\bar{V}\in\mathcal{C}([0,T], B^{s}_{2,r})
\cap\mathcal{C}^{1}([0,T],B^{s-1}_{2,r})$ in (\ref{R-E95}), which
is, of course, a solution of (\ref{R-E10})-(\ref{R-E11}).

Over all, we conclude that $\hat{V}(x,t)$ is the solution of
(\ref{R-E10})-(\ref{R-E11}) satisfying
$\hat{V}-\bar{V}\in\mathcal{C}([0,T], B^{s}_{2,r})
\\ \cap\mathcal{C}^{1}([0,T],B^{s-1}_{2,r})$ for $s>0$, furthermore, we
arrive at $\hat{V}-\bar{V}\in\widetilde{\mathcal{C}}_{T}(
B^{s}_{2,r}) \cap\widetilde{\mathcal{C}}^{1}_{T}(B^{s-1}_{2,r})$.

Finally, the uniqueness is merely a consequence of (\ref{R-E83}).
This completes the proof of Proposition \ref{prop6.2}.
\end{proof}

\section*{Acknowledgments}
J. Xu is partially supported by the NSFC (11001127), China
Postdoctoral Science Foundation (20110490134) and Postdoctoral
Science Foundation of Jiangsu Province (1102057C). He would like to
thank Professor Kawashima for his enthusiastic communication and
hospitality. The second author (S. K.) is partially supported by
Grant-in-Aid for Scientific Research (A) 22244009.

\end{document}